\documentclass[a4paper]{amsart}

\makeatletter
\renewcommand{\tocsection}[3]{%
  \indentlabel{\@ifnotempty{#2}{\bfseries\ignorespaces#1 #2\quad}}\bfseries#3}
\renewcommand{\tocsubsection}[3]{%
  \indentlabel{\@ifnotempty{#2}{\ignorespaces#1 #2\quad}}#3}

\newcommand\@dotsep{4.5}
\def\@tocline#1#2#3#4#5#6#7{\relax
  \ifnum #1>\c@tocdepth 
  \else
    \par \addpenalty\@secpenalty\addvspace{#2}%
    \begingroup \hyphenpenalty\@M
    \@ifempty{#4}{%
      \@tempdima\csname r@tocindent\number#1\endcsname\relax
    }{%
      \@tempdima#4\relax
    }%
    \parindent\z@ \leftskip#3\relax \advance\leftskip\@tempdima\relax
    \rightskip\@pnumwidth plus1em \parfillskip-\@pnumwidth
    #5\leavevmode\hskip-\@tempdima{#6}\nobreak
    \leaders\hbox{$\m@th\mkern \@dotsep mu\hbox{.}\mkern \@dotsep mu$}\hfill
    \nobreak
    \hbox to\@pnumwidth{\@tocpagenum{\ifnum#1=1\bfseries\fi#7}}\par
    \nobreak
    \endgroup
  \fi}
\AtBeginDocument{%
\expandafter\renewcommand\csname r@tocindent0\endcsname{0pt}
}
\def\l@subsection{\@tocline{2}{0pt}{2.5pc}{5pc}{}}
\makeatother

\usepackage{xcolor}
\definecolor{lgray}{RGB}{240,240,240}
\definecolor{webgreen}{rgb}{0,.5,0}
\definecolor{webbrown}{rgb}{.6,0,0}
\definecolor{RoyalBlue}{cmyk}{1, 0.50, 0, 0}
\usepackage[colorlinks=true, breaklinks=true, urlcolor=webbrown, linkcolor=RoyalBlue, citecolor=webgreen, backref=page]{hyperref}

\usepackage{amssymb,mathtools,esint}
\mathtoolsset{showonlyrefs}
\usepackage{epsfig, graphicx} 
\usepackage{verbatim, setspace}

\usepackage{newtxtext}

\usepackage{mathabx}

\newtheorem{theorem}{Theorem}[section]

\newtheorem{proposition}[theorem]{Proposition}
\newtheorem{lemma}[theorem]{Lemma}

\newtheorem*{totik}{Theorem (Totik, \cite[Theorem~14.4]{Totik})}

\newcommand{\T}		{\mathbb{T}}
\newcommand{\D}		{\mathbb{D}}

\newcommand{\C}		{\mathbb{C}}
\newcommand{\N}		{\mathbb{N}}
\newcommand{\Z}		{\mathbb{Z}}

\newcommand{\dist}{\mathrm{dist}}

\newcommand{\supp}{\mathrm{supp}}

\newcommand{\re}{\mathrm{Re}}
\newcommand{\im}{\mathrm{Im}}

\newcommand{\qasq}{\quad \text{as} \quad}
\newcommand{\qandq}{\quad \text{and} \quad}

\newcommand{\ic}{\mathrm{i}}
\renewcommand{\O}{\mathcal{O}}

\numberwithin{equation}{section}

\begin{document}

\title[]{Asymptotics of Polynomials Orthogonal on an Interval with Varying Weights}

\author{Sergey A. Denisov}
\address{Department of Mathematics, University of Wisconsin-Madison, 480n Lincoln Dr., Madison, WI 53706, USA}
\email{\href{mailto:denissov@math.wisc.edu}{denissov@math.wisc.edu}}

\author{Maxim L. Yattselev}

\address{Department of Mathematical Sciences, Indiana University Indianapolis, 402~North Blackford Street, Indianapolis, IN 46202}

\email{\href{mailto:maxyatts@iu.edu}{maxyatts@iu.edu}}

\thanks{The research of the first author was supported by NSF grants DMS-2054465 and DMS-2450716, the Simons Fellowship in Mathematics, the Simons Travel Support for Mathematicians Award, and the Van Vleck Professorship Research Award. The second author completed this work during his tenure as the Royal Society Wolfson Visiting Fellow (RSWVF\textbackslash R3\textbackslash 253003) at the School of Mathematics in Bristol University.
}


\begin{abstract}
We study strong asymptotics of orthogonal polynomials on an interval with varying weights, extending the framework of Totik's theorem on weighted orthogonal polynomials. Our results generalize this theorem in several directions. In particular, we allow the supports of the associated equilibrium measures to converge to a proper subinterval of the original interval of orthogonality or even collapse to a point. These extensions are motivated by applications to multiple orthogonality where such varying measures arise naturally.

\smallskip

\end{abstract}


\maketitle

\tableofcontents

\section{Introduction}
In this paper we study strong (Szeg\H{o}) asymptotics of orthogonal polynomials on an interval with respect to varying weights.  Such problems arise naturally in several settings, in particular, in the theory of multiple orthogonal polynomials. Our primary motivation comes from this connection.\smallskip

 Recall that measures \( \mu_1,\ldots,\mu_d \) on the real line form an Angelesco system if \( \Delta(\mu_i)\cap\Delta(\mu_k) = \emptyset \) for each \( i\neq k \), where \( \Delta(\mu) \) denotes the convex hull of the support of a measure \( \mu \). Given a multi-index  \( \vec n=(n_1,\ldots,n_d) \), \( n_i\geq0 \), the corresponding type~II multiple orthogonal polynomial   is a monic polynomial \( P_{\vec n} \) of degree \( n_1+\cdots+n_d \) characterized by orthogonality conditions
\[
\int x^k P_{\vec n}(x)d\mu_i(x) =0, \quad k\in\{0,\ldots,n_i-1\},  \quad i\in\{1,\ldots,d\}.
\]
A classical theorem of Angelesco (see \cite{Ang19}) states that  \( P_{\vec n} \) has exactly \( n_i \) zeros on \( \Delta_i \). Hence, it can be written in the form \( P_{\vec n}=P_{\vec n,1} \cdots P_{\vec n,d} \), where \( P_{\vec n,i} \) is a monic polynomial of degree \( n_i \) with all of its zeros contained in the interval \( \Delta(\mu_i) \). This factorization transforms multiple orthogonality into an ordinary orthogonality problem with a varying weight. Indeed, it holds that
\begin{align}
\label{weighted ortho}
\int x^k P_n(x)w_n^{2n}(x)d\mu(x) =0, \quad k\in\{0,\ldots,n-1\},
\end{align}
with \( n=n_i \), \( \mu=\mu_i \), \( P_n = P_{\vec n,i} \), and \( w_n^{2n} = \prod_{k\neq i}P_{\vec n,k} \) for any \( i\in\{1,\ldots,d\} \). Denote by $V^\omega$ the logarithmic potential of a measure \( \omega \):
 \[ V^\omega(z):= \int\log|z-x|^{-1}d\omega(x)\,. \]
If \( \nu_n \) is the normalized zero counting measure of  \( \prod_{k\neq i}P_{\vec n,k} \), then
\[
w_n(x) = \exp\left( -\frac Nn V^{\nu_n}(x)\right), \quad N=\frac12\sum_{k\neq i} n_k.
\] Hence, polynomials \( P_{\vec n,i} \) can be viewed as ordinary orthogonal polynomials with non-trivial and varying weights. In this paper, we investigate strong asymptotics of such polynomials. Then, in a follow-up paper 
we intend to use these results to derive strong asymptotics of the polynomials \( P_{\vec n,i} \) themselves. Throughout the text, we denote by  \( T_n(w_n^{2n}\mu) \) the monic polynomials satisfying \eqref{weighted ortho}, where, with a slight abuse of notation, we write \( f\mu \) for the measure \( fd\mu \).

The asymptotic theory of orthogonal polynomials with fixed weights is classical; see the monographs
\cite{Szego,Nevai,StahlTotik,Ismail,Simon,Simon2}. Orthogonal polynomials with varying weights on a finite interval were studied in \cite{Totik}, see also \cite{CYLL98} for the special case of reciprocal polynomial weights. The starting point for our work is \cite[Theorem~14.4]{Totik}. Before we state it, let us describe in detail all the necessary ingredients. 

Given a closed interval \( \Delta=[\alpha,\beta] \), we denote by \( L^p(\omega_\Delta) \) the space of real-valued functions whose moduli are \( p \)-summable with respect to
\begin{align}
\label{arcsine}
d\omega_\Delta(x) := \frac{dx}{\pi\sqrt{(x-\alpha)(\beta-x)}},
\end{align}
the arcsine distribution on \( \Delta \).  We further denote by \( H^2(D_\Delta) \) the Hardy space of functions holomorphic in \( D_\Delta := \overline\C\setminus\Delta \) whose squared moduli admit harmonic majorants in \( D_\Delta \), see \cite[Chapter~10, p.168]{Duren}. This definition is conformally invariant, meaning that \( f \in H^2(\D) \), the standard Hardy space on the unit disk, if and only if \( f\circ\phi_\Delta\in H^2(D_\Delta) \), where
\begin{equation}
\label{phi-w}
\phi_\Delta(z) := \frac2{\beta-\alpha}\left(z - \frac{\beta+\alpha}{2} - w_\Delta(z) \right) \qandq w_\Delta(z) := \sqrt{(z-\alpha)(z-\beta)}\,.
\end{equation}
The branch in the definition of \(w_\Delta\) is taken to be holomorphic off \( \Delta \) and normalized so that \( w_\Delta(z) = z+\mathcal O(1) \) as \( z\to\infty \) (\( \phi_\Delta \) is simply the conformal map of \( D_\Delta \) onto $\D$ such that \( \phi_\Delta(\infty) =0 \) and \( \phi_\Delta(\beta)=1 \)). Every function \( g\in H^2(D_\Delta) \) possesses non-tangential limits \( g_\pm \) from above and below \( \Delta \) that satisfy
\[
g_\pm \in L^2(\omega_\Delta) \qandq \log |g_\pm| \in L^1(\omega_\Delta).
\]
We shall say that \( G \) is an \emph{outer function} in \( H^2(D_\Delta) \) if \( G\circ \phi_\Delta^{-1} \) is an outer function in \( H^2(\D) \), see \cite[Section~2.4]{Duren}. For instance, given a nonnegative function \( f \in L^2(\omega_\Delta) \) such that \( \log f \in L^1(\omega_\Delta) \), the function
\begin{equation}
\label{outer}
\Omega_\Delta(f;z) := \exp\left( w_\Delta(z) \int_\Delta \log f(x) \frac{d\omega_\Delta(x)}{z-x} \right)
\end{equation}
is an outer function in \( H^2(D_\Delta) \). It also satisfies $\Omega_\Delta(f;\infty)>0$ and is conjugate-symmetric. For such functions it holds that
\begin{equation}
\label{outer-b}
|\Omega_{\Delta\pm}(f;x)| = f(x) \quad \text{for almost every} \quad x\in\Delta.
\end{equation}
In particular, any conjugate-symmetric outer function in \( H^2(D_\Delta) \) that is positive at infinity can be recovered through the modulus of its boundary values via \eqref{outer}--\eqref{outer-b}. Notice also that
\begin{equation}
\label{geom-mean}
\log \Omega_\Delta(f;\infty) = \int_\Delta \log fd\omega_\Delta.
\end{equation}

Given a  measure \( \mu \) on the real line, we write
\begin{equation}
\label{Lebesgue}
d\mu(x) = \mu^\prime(x)dx + d\mu^s(x),
\end{equation}
where \( \mu^\prime(x) \) is the Radon-Nikodym derivative of \( \mu \) with respect to the Lebesgue measure on the real line and \( \mu^s\) is singular to the Lebesgue measure. We denote by \( |\mu| \) the total mass of \( \mu \).

We say that \( \mu \) is a \emph{Szeg\H{o}  measure} on \( \Delta\subseteq\Delta(\mu) \), written \( \mu\in\mathrm{Sz}(\Delta) \), if \( \log\mu^\prime \in L^1(\omega_\Delta) \). In this case, \( \log v_\Delta \in L^1(\omega_\Delta) \), where
\begin{equation}
\label{szego}
d\mu_{|\Delta}(x) = v_\Delta(x)d\omega_\Delta(x) + d\mu_{|\Delta}^s(x).
\end{equation}
Equivalently, \( v_\Delta(x) = \pi\mu^\prime(x)\sqrt{(x-\alpha)(\beta-x)}, \, x\in\Delta \).  In the classical theory of strong asymptotics, one usually assumes that \( \Delta=\Delta(\mu) \) in \eqref{szego}. 
Our primary motivation, however, comes from multiple orthogonality, where it is essential to work with restrictions of $\mu$ to proper subintervals of its support. This is the reason for writing \( \mu_{|\Delta} \) in \eqref{szego}. When \( \mu\in\mathrm{Sz}(\Delta) \), one can define the associated \emph{Szeg\H{o} function} by
\begin{equation}
\label{SzegoFun}
G(\mu_{|\Delta};z) := \Omega_\Delta\big(\sqrt{v_\Delta};z\big), \quad z\in D_\Delta.
\end{equation}
By construction, $G(\mu_{|\Delta};\cdot)$ is an outer function in \( H^2(D_\Delta) \) and its boundary values on \( \Delta \) satisfy \( |G_\pm(\mu_{|\Delta};x)|^2 = v_\Delta(x) \) for almost every \( x\in \Delta \). In what follows, we remove the subscript \( |\Delta \) from \( \mu \) in \eqref{SzegoFun} if \( \Delta=\Delta(\mu) \). 

Let \( w \) be a continuous nonnegative function on \( \Delta \) that is non-zero quasi-everywhere on \( \Delta \) (that is, outside a set of zero logarithmic capacity). Among all probability  measures \( \omega \) on \( \Delta \), there exists a unique one, denoted $\omega_w$, that minimizes the weighted energy  functional 
\[
I_w[\omega] := \iint \log\big[ |x-y| w(x)w(y)\big]^{-1}d\omega(x)d\omega(x).
\]
This measure is known as \emph{extremal measure} (or \emph{weighted equilibrium measure/distribution}), see \cite[Theorem~A]{Totik}. Let \[ F_w:= I_w[\omega_w] + \int\log wd\omega_w.\] It is known that
\begin{align}
\label{equilibrium condition}
V^{\omega}(x) 
\begin{cases}
\geq \log w(x) + F_w, & x\in\Delta, \\
\leq \log w(x) + F_w, & x\in\supp(\omega_w).
\end{cases}
\end{align} 
We are now ready to state the theorem that serves as the starting point for our work (we provide its reduced version for the asymptotics of monic polynomials).
\begin{totik}
\label{thm:totik}
Let \( \{w_n\} \) be a sequence of continuous, nonnegative functions on \([-1,1]\) that are non-vanishing quasi-everywhere. Assume that the corresponding extremal measures \( \omega_{w_n} \) are supported on the entire interval \( [-1,1] \), are absolutely continuous, and satisfy\footnote{By \( f(x)\lesssim g(x)\), \( x\in S \), we mean that \( f(x)\leq cg(x)\), \( x\in S \), for some constant \( c>0 \).}
\[
(1-x^2)^{\varkappa_L} \lesssim \omega_{w_n}^\prime(x) \lesssim (1-x^2)^{\varkappa_U}, \quad x\in[-1,1],
\]
for some constants \( -1<\varkappa_U\leq\varkappa_L \). Furthermore, let \( \mu\in\mathrm{Sz}([-1,1]) \) be an absolutely continuous measure. Then, the asymptotics 
\[
T_n(w_n^{2n}\mu)(z) = (1+o(1))\exp\left( n \int\log(z-x)d\omega_{w_n}(x)\right) \frac{G(\mu;\infty)}{G(\mu;z)}
\]
holds locally uniformly in \( D_{[-1,1]} \).
\end{totik}

Observe that Totik's theorem requires \( \supp(\omega_n) = [-1,1] \).
In this setting, relation \eqref{equilibrium condition}  shows that the weight of orthogonality may equivalently be written as
 \( e^{2nV^{\omega_{w_n}}}\mu \)
since multiplying the weight by a positive constant does not change the corresponding monic orthogonal polynomial.  For the applications that motivate our work, however, this result is not sufficiently general. One objective of the present paper is therefore to extend Totik's theorem in several ways.

  In Theorem~\ref{thm:3}, we replace a single absolutely continuous measure \( \mu \) with a sequence of not necessarily absolutely continuous measures \( e^{h_n}\mu_n \), where functions \( h_n \) belong to a compact subset of continuous functions on \( [-1,1] \) (in fact, the theorem is formulated for a general interval \( \Delta \)).
  Totik's proof of the above theorem relies on \cite[Theorem~10.2]{Totik} together with results from \cite{CYLL98} concerning orthogonal polynomials with reciprocal polynomial weights. We follow the same general strategy. The main difference is that we replace the key results from \cite{CYLL98} by our Theorem~\ref{thm:2}, which is itself derived from its analogue on the unit circle, namely Theorem~\ref{thm:1}.
  
  As discussed earlier, our applications also require treating situations in which the supports of the extremal measures \( \omega_{w_n} \)  are proper subsets of the original interval of orthogonality. This leads to two additional extensions.  In Theorem~\ref{thm:4}, we address the situation where \( \supp(\omega_{w_n}) \) converge to a non-degenerate interval and in Theorem~\ref{thm:5} we deal with the case where \( \supp(\omega_{w_n}) \) collapse into a point. Theorems~\ref{thm:4} and~\ref{thm:5} require us to strengthen the notion of a Szeg\H{o} measure. Accordingly, in Section~\ref{sec:usm} we introduce the notion of a {\it uniformly Szeg\H{o} measure} and in Section~\ref{sec:ssm} we  define the class of {\it strongly Szeg\H{o} measures with index \( \gamma \).}

\section{OPs with Reciprocal Polynomial Weights}

Let \(  \sigma \) be a measure on the unit circle \( \T \).  A sequence of polynomials \( \phi_n \), \( n\geq0 \), is said to be orthonormal  with respect to \( \sigma \) if \( \deg(\phi_n) =n \) and
\begin{align}
\int \phi_n(\xi)\overline{\phi_m(\xi)}d\sigma(\xi) = \delta_{nm},
\end{align}
where \(  \delta_{nm} \) denotes the usual Kronecker symbol. As standard, we assume that the leading coefficient of \( \phi_n \) is positive.  Given a polynomial \( p \) of degree at most \( n \), denote by \( p^* \) its  reciprocal polynomial (w.r.t. index \( n \)). That is,
\begin{align}
p^*(z) = z^n \overline{p_n(1/\overline z)}.
\end{align}
We are interested in the asymptotic behavior of the polynomials \( \phi_n^* \). It is known \cite[Theorem~11.4.1]{Szego} that the zeros of \( \phi_n \) belong to the unit disk \( \D \)  and consequently the zeros of \( \phi_n^* \) belong to the exterior of the closed unit disk. A  celebrated result of Szeg\H{o}, see \cite[Theorem~12.1.1]{Szego}, states that if \( d\sigma = \upsilon dm \), where \( dm = (2\pi)^{-1}|d\xi| \) is the normalized arclength measure on $\T$, and \( \log\upsilon\in L^1(\T) \), then
\begin{align}
\label{szego asymptotics}
\phi_n^*(z) D(\sigma;z) = 1+ o(1) \qasq n\to\infty
\end{align}
locally uniformly in the unit disk \( \D \), where \( D(\sigma;z) \) is the Szeg\H{o} function of the measure \( \sigma \) defined by
\begin{align}
\label{szego function}
D(\sigma;z):=\exp\left(\frac 12\int\frac{\xi+z}{\xi-z}\log\upsilon(\xi)dm(\xi)\right), \quad z\not\in\T.
\end{align}
The function \( D(\sigma;z) \) is analytic in \( \overline\C\setminus\T \). In fact, it is an outer function in both \( \D \) and \( \overline\C\setminus\overline \D \), its values inside and outside of the unit disk are related via the identity
\begin{align}
D^{-1}(\sigma;z) = \overline{D(\sigma;1/\bar z)}, \quad z\not\in \T\,.
\end{align}
It has a non-tangential limit on \( \T \) (taken within $\D$) that satisfies  $|D(\sigma,\xi)|^2=\upsilon(\xi)$ for almost every $\xi$ on $\T$.  

Verblunsky and Kolmogorov observed that \eqref{szego asymptotics} does not change if \( \sigma \) is of the more general form \( d\sigma = \upsilon dm + d\sigma^s \), where \( d\sigma^s \) is singular with respect to \( dm \). In this case, the Szeg\H{o} function of such measures is still defined by \eqref{szego function}. A further extension of \eqref{szego asymptotics} was obtained in \cite{CYLL98} (see also \cite{St00}), where \( \sigma \) was modified by an \( n \)-dependent reciprocal polynomial weight. Such weights arise naturally in the theory of multi-point Pad\'e approximants, while we need them as a technical tool in the proof of Theorem~\ref{thm:3}. Our applications, however, require a more general formulation which we discuss below.

Accordingly, let
 \( (\sigma_n,g_n,W_n) \)
be a sequence of triples consisting of a finite measure on \(\T\), a real-valued continuous function on \(\T\), and a monic polynomial \(W_n\) of degree $n$ whose zeros all lie in \(\D\).
 For each \( n \), define the  inner product:
\[
\langle f,k\rangle_{\T,n} := \int f(\xi)\overline{k(\xi)} \, \frac{e^{g_n(\xi)}d\sigma_n(\xi)}{|W_n(\xi)|^2}.
\]
We are interested in the orthonormal polynomials \( \phi_n \), \( \deg\phi_n=n \), satisfying
\begin{equation}
\label{ortho-circle}
\begin{cases}
\langle \phi_n,\xi^k\rangle_{\T,n}=0, &  k\in\{0,1,\ldots,n-1\}, \smallskip \\
\langle \phi_n,\phi_n\rangle_{\T,n} =1,
\end{cases}
\end{equation}
and normalized to have a positive leading coefficient, i.e., \( \phi_n(z) = \alpha_n z^n + \ldots \), \( \alpha_n>0 \) (the special case \( g_n=0 \) and \( \sigma_n=\sigma \)  is exactly the setting considered in \cite{St00}).  If we denote the zeros of \( W_n \) by \( b_{n,i} \), \( i\in\{1,2,\ldots,n\} \), then
\[
W_n(z)=\prod_{i=1}^n (z-b_{n,i}) \qandq W_n^*(z)=\prod_{i=1}^n (1-\overline{b_{n,i}}z).
\]
Recall that all \(b_{n,i}\) belong to the unit disk. 
Write $d\sigma_n=\upsilon_ndm+d\sigma_n^s$, where the measures \( \sigma_n^s \) are singular to $m$. Assume that
\emph{
\begin{itemize}
\item[$(A_\T)$] there exists a finite measure $d\sigma=\upsilon dm+d\sigma^s$, where \( \sigma^s \) is singular to $m$, such that
\[
\limsup_{n\to\infty} \int fd\sigma_n \leq \int fd\sigma;
\]
for every nonnegative continuous function $f$ on $\T$;
\item[$(B_\T)$] the functions \( \log \upsilon_n \) and \(\log  \upsilon \) belong to \( L^1(\T)\) and  \( \|\log \upsilon_n- \log \upsilon\|_{L^1(\T)}\to 0 \) as \( n\to\infty \); \smallskip
\item[$(C_\T)$] the real-valued functions \( g_n \) belong to \( \mathcal E \), a fixed compact subset of \( C(\T) \), the space of continuous functions on the unit circle; \smallskip
\item[$(D_\T)$] the zeros of $W_n$ satisfy \( \sum_{i=1}^n (1-|b_{n,i}|)\to \infty \) as \( n\to\infty \).
\end{itemize}
}
It is known, see \cite[Section~II.2]{Ga}, that condition $(D_\T)$ is equivalent to
\begin{equation}
\label{Blaschke}
\frac{W_n(z)}{W_n^*(z)} = o(1) \qasq n\to\infty
\end{equation}
locally uniformly in \( \D \). 

\begin{theorem}
\label{thm:1}
Let \( \phi_n \) be the orthonormal polynomials defined by \eqref{ortho-circle}, where the triples \( (\sigma_n,g_n,W_n) \) satisfy conditions \( (A_\T) 
\)--\((D_\T) \). Then,
\begin{equation}
\label{A1-1}
 \frac{\phi_n^*(z)}{W_n^*(z)}D_n(z)  = 1 + o_\mathcal{E}(1) \qandq \frac{\phi_n(z)}{\phi_n^*(z)} = o_\mathcal{E}(1) \qasq n\to\infty
 \end{equation}
locally uniformly in \( \D \), where \( D_n(z) := D(e^{g_n}\sigma_n;z) \). In particular,
\begin{equation}
\label{A1-2}
\alpha_n D_n(0) = 1 + o_\mathcal{E}(1) \qasq n\to\infty.
\end{equation}
\end{theorem}

It is well-known, see \cite[Theorem~11.5]{Szego}, that orthogonal polynomials on \( [-1,1] \) can be expressed through associated orthogonal polynomials on the unit circle. In particular, one can carry over the Szeg\H{o} asymptotics obtained for polynomials orthogonal on the unit circle to polynomials orthogonal on the interval, see \cite[Theorem~12.1.2]{Szego} and \cite{St00}.
This is exactly what we do next. That is, we translate Theorem~\ref{thm:1} into an analogous theorem about polynomials orthogonal on \( [-1,1] \). To this end,   we look at triples \( (\tilde\mu_n,h_n,\tau_n) \), where \( \tilde\mu_n \) is a finite  measure, \( h_n \) is a continuous function, and \( \tau_n \) is a polynomial of degree at most $2n$ with real coefficients that does not vanish on \( [-1,1] \) and is normalized by $\tau_n(0)=1$.  For each \( n \) we define an inner product 
\[
\langle f,g\rangle_{[-1,1],n} := \int f(x)\overline{g(x)} \, \frac{e^{h_n(x)}d\tilde\mu_n(x)}{\tau_n(x)}.
\]
We consider the corresponding orthonormal polynomials $p_n$ characterized by  \( \deg p_n=n \) and
\begin{equation}
\label{ortho-interval}
\begin{cases}
\langle p_n,x^k\rangle_{[-1,1],n}=0, &  k\in\{0,1,\ldots,n-1\}, \smallskip \\
\langle p_n,p_n\rangle_{[-1,1],n} =1,
\end{cases}
\end{equation}
with the normalization that the leading coefficient  \( \gamma_n \) is positive, i.e., \(\gamma_n>0\) where \( p_n(z) = \gamma_n z^n +\ldots \). The assumptions we made about the polynomials \( \tau_n \) can be equivalently stated in the following way. Let \( \{a_{n,1},a_{n,2},\ldots,a_{n,2n} \} \) be a conjugate-symmetric multi-set (points \( a_{n,i} \) can coincide and be either real or come in complex-conjugate pairs) such that \( a_{n,i}\in \overline{\C}\backslash [-1,1] \)). Then,
\[
\tau_n(x) = \prod_{i=1}^{2n}\left(1-\frac{x}{a_{n,i}}\right),
\]
where we interpret \( x/a_{n,i} \) as  \( 0 \) when \( a_{n,i} = \infty \). Below, we assume the following hypothesis:
\emph{
\begin{itemize}
\item[$(A_{[-1,1]})$] there exists a finite measure $\mu$ on $[-1,1]$ such that 
\[
\limsup_{n\to\infty} \int fd\tilde\mu_n \leq \int fd\mu
\]
for every nonnegative continuous function $f$ on $[-1,1]$;
\item[$(B_{[-1,1]})$] it holds that \( \|\log \tilde v_n- \log v\|_{L^1(\omega_{[-1,1]})}\to 0 \) as \( n\to\infty \), where \( \tilde v_n \) and \( v \) denote the Radon-Nikodym derivatives of \( \tilde\mu_n \) and \( \mu \) with respect to \( \omega_{[-1,1]} \), see \eqref{arcsine} and \eqref{szego}; \smallskip
\item[$(C_{[-1,1]})$] the functions \( h_n \) belong to \( \mathcal K \), a fixed compact subset of \( C[-1,1] \), the space of continuous functions on \( [-1,1] \); \smallskip
\item[$(D_{[-1,1]})$] the zeros \( \{a_{n,i}\} \) of the polynomials \( \tau_n \) satisfy \( \sum_{i=1}^{2n} (1-|\phi(a_{n,i})|)\to \infty \) as \( n\to\infty \), where \( \phi = \phi_{[-1,1]} \) is the conformal map defined in \eqref{phi-w}.
\end{itemize}
}

\begin{theorem} 
\label{thm:2} 
Let \( p_n \) be the orthonormal polynomials defined  in \eqref{ortho-interval}, where the triples $(\tilde\mu_n,h_n,\tau_n)$ satisfy conditions \( (A_{[-1,1]}) \)--\( (D_{[-1,1]}) \). Then,
\begin{equation}
\label{A8-1}
2\widetilde G_n^2(z) \, \frac{p_n^2(z)}{\tau_n(z)} \, \prod_{j=1}^{2n} \frac{\phi(z)-\phi(a_{n,j})}{1-\overline{\phi(a_{n,j})}\phi(z)} = 1 + o_{\mathcal K}(1) \qasq n\to\infty
\end{equation}
locally uniformly in $D_{[-1,1]}$, where \( \widetilde G_n(z)=G(e^{h_n}\tilde\mu_n;z) \). In particular,
\begin{equation}
\label{A8-2}
\gamma^2_n\widetilde G_n^2(\infty)2^{1-2n} \, \prod_{i: a_{n,i}\neq \infty} \big(2a_{n,i}\phi(a_{n,i})\big) = 1 + o_{\mathcal K}(1) \qasq n\to\infty.
\end{equation}
\end{theorem}

\subsection{Proof of Theorem~\ref{thm:1}}

The proof closely follows the argument in \cite{St00}, although several modifications are required to accommodate the greater generality of our assumptions. For clarity, we divide the proof into a sequence of lemmas.\smallskip

 We define the Caratheodory function of a measure \( \sigma \) by
\[
F(\sigma;z):=\int \frac{\xi+z}{\xi-z}d\sigma(\xi).
\]
Clearly, \( \re(F(\sigma;\cdot)) \) is the Poisson integral of \( \sigma \) and therefore is a positive harmonic function in \( \D \). From \cite[Theorem~1.2.4]{Ransford} we know that every function \( \lambda \), which is harmonic  in \( \D \) and extends continuously to \( \overline\D \), admits repesentation \( \lambda = \re(F(\lambda m;\cdot)) \) and \( \lambda(0)=|\lambda m|\). For each orthonormal polynomial \( \phi_n \), we further introduce  its companion polynomial \( \psi_n \). This is the polynomial of degree at most \( n \) such that \( \psi_n^* \) interpolates  $\phi_n^*(z) F(e^{g_n}\sigma_n;z)$ at the zeros of $zW_n(z)$. This polynomial has an explicit integral representation, see \cite[Equation~(3.6)]{St00}, which fixes its normalization.

\begin{lemma}
\label{lem:pw2}
Under the conditions of Theorem~\ref{thm:1}, we have
\[
F(e^{g_n}\sigma_n;z)- \frac{\psi^*_n(z)}{\phi^*_n(z)} = o_\mathcal{E}(1) \qasq n\to\infty
\]
locally uniformly in \( \D \). Moreover, 
\begin{equation}\label{zar1}
F(e^{g_n}\sigma_n;0)=\frac{\psi^*_n(0)}{\phi^*_n(0)}\,.
\end{equation}
\end{lemma}
\begin{proof}
Combining formulas (3.23), (3.24), and (3.26) from \cite{St00}, we get
\[
\left|F(e^{g_n}\sigma_n;z)- \frac{\psi^*_n(z)}{\phi^*_n(z)} \right| \leq 2\sqrt 2\left|z\frac{W_n(z)}{W_n^*(z)}\right| \frac{|e^{g_n}\sigma_n|}{(1-|z|)^{3/2}}, \quad z\in \D,
\]
where \( |\mu| \) denotes the total mass of the measure \( \mu \). Taking $f=1$ in $(A_\T)$, we get
\begin{equation}
\label{zar2}
\limsup_{n\to \infty}|\sigma_n|\le|\sigma| \qandq \limsup_{n\to \infty}|e^{g_n}\sigma_n|\lesssim_{\mathcal E} |\sigma|,
\end{equation}
where the last bound follows from $(C_\T)$. Applying \eqref{Blaschke} finishes the proof of the lemma.
\end{proof}

Notice that Jensen's inequality and $(B_\T)$ yield that
\begin{equation}
\label{zar3}
\liminf_{n\to\infty} |\sigma_n|\ge \liminf_{n\to\infty} \exp\left(\int \log\upsilon_n dm\right)=\exp\left(\int \log\upsilon dm\right)>-\infty
\end{equation}
and $(C_\T$) gives a lower bound
\begin{equation}
\label{zar4}
\liminf_{n\to\infty} |e^{g_n}\sigma_n|\gtrsim_{\mathcal{E}}\exp\left(\int \log\upsilon dm\right).
\end{equation}

Let \( \lambda_n \) be a harmonic function in some neighborhood of \( \overline \D \) defined by
\begin{equation}
\label{lam1}
\lambda_n(z) := \re\left( \frac{\psi^*_n(z)}{\phi^*_n(z)} \right)
\end{equation}
(recall that \( \phi^*_n \) has no zeros in \( \overline \D\)). It has been shown in \cite[Equation~(3.9)]{St00} that
\begin{equation}
\label{lam2}
\frac{\psi^*_n(z)}{\phi^*_n(z)} = \int \frac{\xi+z}{\xi-z} \left| \frac{W_n(\xi)}{\phi_n(\xi)} \right|^2 dm(\xi).
\end{equation}
Hence, \( \lambda_n \) is a Poisson integral of an absolutely continuous measure with a strictly positive density and therefore is a strictly positive harmonic function on the closed unit disk.  It readily follows from the second claim of Lemma~\ref{lem:pw2} that
\begin{equation}
\label{ls-masses}
|\lambda_nm| = \lambda_n(0) = F(e^{g_n}\sigma_n;0) = |e^{g_n}\sigma_n|.
\end{equation}

\begin{lemma}
\label{lem:pw3} 
Let $\mathcal Z$ be a compact set in $C(\T)$. Under the conditions of Theorem~\ref{thm:1}, it holds that for any \( \epsilon>0 \) there exists \( N_{\mathcal E,\mathcal Z}(\epsilon) \) such that
\[
\left|\int he^{g_n}d\sigma_n-\int h\lambda_ndm\right| \leq \epsilon
\]
for all \( n \geq N_{\mathcal E,\mathcal Z}(\epsilon) \) and each \( h\in\mathcal Z \).
\end{lemma}
\begin{proof}
We use a standard approximation argument. Recall that
\[
F(\sigma;z) = |\sigma| + 2\sum_{i=1}^\infty z^i \int \xi^{-i} d\sigma(\xi) = |\sigma| + 2\sum_{i=1}^\infty z^i \overline{\int \xi^i d\sigma(\xi)}.
\]
Since locally uniform convergence of analytic functions implies convergence of their Taylor coefficients, we can use   Lemma~\ref{lem:pw2} and \eqref{lam2} to conclude that
\[
\int h \big(e^{g_n}d\sigma_n-\lambda_ndm\big)  = o_{\mathcal E,h}(1) \qasq n\to\infty
\]
for every monomial \( h(\xi) = \xi^i \), $i\in \Z$. By linearity, the same conclusion holds for every trigonometric polynomial $h$. Now, given $\epsilon>0$, we can use compactness and denseness of trigonometric polynomials in $C(\T)$ to find a finite collection of trigonometric polynomials \( \{h_1,h_2,\ldots,h_{K(\epsilon)}\}\) such that for each $h\in \mathcal{Z}$ there is $k\in\{1,2,\ldots,K(\epsilon)\}$ for which\
\[
\|h-h_k\|_\infty<\frac{\epsilon}{3M_{\mathcal E}},
\]
where \(M_{\mathcal E}\) is chosen so that \( |e^{g_n}\sigma_n| \leq M_{\mathcal E} \) for all \( n \). Then, 
\[
 \left|\int (h-h_k)e^{g_n}d\sigma_n\right| \leq \frac\epsilon3 \qandq \left|\int (h-h_k)\lambda_ndm\right|\leq \frac\epsilon3
\]
by \eqref{ls-masses}. As just observed, for each \( h_k \) one can find a natural number $N_{\mathcal E,h_k}(\epsilon)$ such that 
\[
\left| \int  h_k\big(e^{g_n}d\sigma_n-\lambda_ndm\big) \right|\le \frac\epsilon3
\]
for $n\ge N_{\mathcal E,h_k}(\epsilon)$. Taking $N_{\mathcal E,\mathcal Z}(\epsilon)=\max_{1\leq k\leq K(\epsilon)} N_{\mathcal E,h_k}(\epsilon)$ yields the desired claim.
\end{proof}

For the next step, we shall need the mutual entropy of two measures. Let $\mu$ and $\nu$ be two measures on $\T$ such that $\mu$ is absolutely continuous with respect to $\nu$. The entropy $S(\mu|\nu)$ is defined as 
\[
S(\mu|\nu)=-\int \log\left(\frac{d\mu}{d\nu}\right)d\mu.
\]
It is known that $S(\mu|\nu)\le \log |\nu|$  and, see \cite[Lemma~2.3.3]{Simon}, it holds that
\begin{equation}
\label{aper1}
S(\mu|\nu) = \inf_f\left(  \int fd\nu-\int (1+\log f)d\mu \right),
\end{equation}
where the infimum is taken over all positive continuous functions on \( \T \). Moreover, see \cite[Example~2.3.2]{Simon}, if \( \mu=m \) and \( d\nu= \nu^\prime dm + d\nu^s \), then
\begin{equation}
\label{aper2}
S(m|\nu) = - \int\log\left(\frac{dm}{d\nu}\right)dm = \int \log\nu^\prime dm.
\end{equation}

\begin{lemma}
\label{lem:pw4}
Under the assumptions of Theorem~\ref{thm:1},  for every \( \epsilon>0 \), there exists \( N_{\mathcal E}(\epsilon) \),
 depending only on the compact set \( \mathcal E \),  such that 
\[
\int \log\lambda_n dm \le \int \big( g_n+\log\upsilon_n\big) dm + \epsilon, \quad n\geq N_{\mathcal E}(\epsilon).
\]
\end{lemma}
\begin{proof}
By \eqref{aper1}, \eqref{aper2} and  condition \( (B_\T) \),  there exists a positive continuous function \( f_\epsilon \) and a natural number \( N_1(\epsilon) \) such that
\[
\int f_\epsilon d\sigma - \int (1+\log f_\epsilon)dm \leq S(m|\sigma) +  \frac\epsilon4 \leq \int\log\upsilon dm +  \frac\epsilon4 \leq \int\log\upsilon_n dm +  \frac\epsilon2
\]
for all \( n\geq N_1(\epsilon) \). Applying \eqref{aper1} and \eqref{aper2} one more time, we use the inequality above to get
\begin{align}
\int \log\lambda_n dm & = S(m|\lambda_nm) \leq \int (e^{-g_n}f_\epsilon) \lambda_n dm - \int \big(1+\log(e^{-g_n}f_\epsilon) \big) dm \nonumber \\
&\leq \int\big( g_n+\log\upsilon_n\big) dm + \int (e^{-g_n}f_\epsilon)\lambda_n dm - \int f_\epsilon d\sigma +  \frac\epsilon2
\label{A4-1}
\end{align}
for all \( n\geq N_1(\epsilon) \). Furthermore,  condition \( (A_\T) \) implies that there exists a natural number \( N_2(\epsilon) \) such that
\[
\int f_\epsilon d\sigma_n \leq \int f_\epsilon d\sigma +  \frac\epsilon4
\]
for all \( n\geq N_2(\epsilon) \). Applying Lemma~\ref{lem:pw3} with \( \mathcal Z=\{e^{-g}f_\epsilon:g\in\mathcal E\} \), we conclude that there exists \( N_{\mathcal E}(\epsilon) \geq \max\{ N_1(\epsilon),N_2(\epsilon) \} \)  such that
\begin{equation}
\label{A4-2}
\int (e^{-g_n}f_\epsilon)\lambda_n dm \leq \int (e^{-g_n}f_\epsilon) e^{g_n}d\sigma_n +  \frac\epsilon4 \leq \int f_\epsilon d\sigma +  \frac\epsilon2
\end{equation}
for all \( n\geq N_{\mathcal E}(\epsilon) \). Clearly, inequalities \eqref{A4-1} and \eqref{A4-2} yield the desired claim.
\end{proof}

\begin{lemma}
\label{lem:pw5}
Under the assumptions of Theorem~\ref{thm:1}, asymptotic relation \eqref{A1-2} holds.
\end{lemma}
\begin{proof}
 It follows directly from \eqref{lam1} and \eqref{lam2} combined with \cite[Theorem~1.2.4]{Ransford}, see also \cite[Equation~(3.8)]{St00}, that
\[
\left|\frac{W_n^*(\xi)}{\phi_n^*(\xi)}\right|^2 = \left|\frac{W_n(\xi)}{\phi_n(\xi)}\right|^2 = \lambda_n(\xi), \quad |\xi|=1.
\]
As \( W_n^*/\phi_n^* \) is analytic and non-vanishing on the closed unit disk, the logarithm of its absolute value is harmonic there. Hence, the mean-value property gives
\[
\log\alpha_n = -\log\left|\frac{W_n^*(0)}{\phi_n^*(0)}\right| = - \frac12\int\log\lambda_ndm,
\]
where we also used \( \phi_n^*(0) = \alpha_n>0 \) and \( W_n^*(0)=1 \).  Applying Lemma~\ref{lem:pw4} together with the definition of \( D_n \), we obtain
\[
\log\alpha_n \geq -\frac12\int \big( g_n+\log\upsilon_n\big) dm + o_{\mathcal E}(1) = - \log D_n(0) + o_{\mathcal E}(1). 
\]
On the other hand, \eqref{ortho-circle} provides
\[
0 = \log \int \left|\frac{\phi_n}{W_n}\right|^2 e^{g_n}d\sigma_n \geq \log \int \left|\frac{\phi_n}{W_n}\right|^2 e^{g_n} \upsilon_ndm.
\]
Applying  Jensen's inequality gives
\[
0 \geq \int \log \frac{e^{g_n}\upsilon_n}{\lambda_n} dm = 2\log\big(\alpha_nD_n(0) \big).
\]
Combining the two estimates finishes the proof of the lemma.
\end{proof}

\begin{lemma}
\label{lem:pw6}
Under the assumptions of Theorem~\ref{thm:1}, it holds that
\[
\int\left|\frac{\phi_n^*}{W_n^*}D_n-1\right|^2dm = o_{\mathcal E}(1) \qasq n\to\infty,
\]
where \( D_n(\xi) \) denotes the non-tangential boundary values of \( D_n(z) \) on \( \T \) taken from within \( \D \). Consequently, the first asymptotic formula in \eqref{A1-1} takes place.
\end{lemma}
\begin{proof}
Denote the integral in the statement of the lemma by \( I \). Expanding the square gives
\[
I = \int \left|\frac{\phi_n^*}{W_n^*}D_n\right|^2dm + 1 -2\int\re \left(\frac{\phi_n^*}{W_n^*}D_n\right)dm.
\]
Recall that \( |D_n(\xi)|^2=e^{g_n(\xi)}\upsilon_n(\xi) \) for almost every \( |\xi|=1 \). By the mean-value property for harmonic functions, one has
\[
I = \int \left|\frac{\phi_n}{W_n}\right|^2 e^{g_n}\upsilon_ndm + 1 - 2\re \left(\frac{\phi_n^*(0)}{W_n^*(0)}D_n(0)\right) \leq 2 -2\alpha_nD_n(0),
\]
where we used that $W^*_n(0)=1$. The first claim of the lemma now follows from \eqref{A1-2}. In particular, we have that the functions \( 1-\phi_n^*D_n/W_n^* \) belong to the Hardy space \( H^2(\D) \). Thus, the second claim of the lemma now follows from the first one and the Cauchy integral formula for functions in \( H^2(\D) \).
\end{proof}

\begin{lemma}
\label{lem:pw7}
Under the assumptions of Theorem~\ref{thm:1}, the second asymptotic formula in \eqref{A1-1} holds.
\end{lemma}
\begin{proof}
It follows from the first asymptotic formula in \eqref{A1-1} that
\[
\frac{\phi_n(z)}{\phi_n^*(z)} = \frac{W_n^*(z)}{\phi_n^*(z)}  \frac{\phi_n(z)}{W_n^*(z)} = (1 +o_{\mathcal E}(1)) \frac{\phi_n(z)D_n(z)}{W_n^*(z)}.
\]
Thus, it is sufficient to study the behavior of \( \phi_nD_n/W_n^* \) in the unit disk. Since these functions belong to $H^2(\D)$ and \( dm(\xi) = d\xi/(2\pi\ic\xi) \),  the Cauchy integral formula gives
\[
\frac{\phi_n(z)D_n(z)}{W_n^*(z)} = \int \frac{\phi_n(\xi)D_n(\xi)}{W_n^*(\xi)} \frac{dm(\xi)}{1-z\overline\xi} = \int B_n(\xi)U_n(\xi) \frac{\phi_n(\xi)\overline{D_n(\xi)}}{W_n(\xi)} \frac{dm(\xi)}{1-z\overline\xi},
\]
where \( B_n = W_n/W_n^* \) and \( U_n(\xi) = D_n(\xi)/\overline{D_n(\xi)} \), \( \xi\in\T \). Since \( |B_nU_n| \equiv 1 \) on \( \T \), it follows from the Cauchy-Schwarz inequality that
\[
\left| \int  \left(\frac{\phi_n(\xi)\overline{D_n(\xi)}}{W_n(\xi)} - 1 \right) B_n(\xi)U_n(\xi)\frac{dm(\xi)}{1-z\overline\xi}\right| \leq \frac{\|\phi_n\overline{D_n}/W_n -1 \|_{L^2(\T)}}{\sqrt{1-|z|^2}}.
\]
Thus, we deduce from the first claim of Lemma~\ref{lem:pw6} that
\[
\frac{\phi_n(z)D_n(z)}{W_n^*(z)} = o_{\mathcal E}(1) + \int B_n(\xi)U_n(\xi) \frac{dm(\xi)}{1-z\overline\xi} = o_{\mathcal E}(1) + \big\langle B_n,u(\cdot;g_n,\upsilon_n,z)\big\rangle,
\]
where \( o_{\mathcal E}(1) \) holds locally uniformly in the unit disk and
\[
u(\xi;g_n,\upsilon_n,z) := \frac1{1-\overline z\xi}\frac{\overline{D(e^{g_n};\xi)}}{D(e^{g_n};\xi)} \frac{\overline{D(\upsilon_n;\xi)}}{D(\upsilon_n;\xi)}.
\]
It follows from \eqref{Blaschke} that \( \langle B_n,u\rangle \to 0 \) as \( n\to\infty \) for any fixed \( u\in L^2(\T) \). Indeed, this property clearly holds for each \( u(\xi)=\xi^i, i\in \{0,1,\ldots\} \), and therefore for the whole space by the density of the monomials. As Blaschke products \( B_n \) have unit norms in \( L^2(\T) \), it also holds that \( \langle B_n,u_n\rangle \to 0 \) as \( n\to\infty \) whenever \( u_n\to u \)  in \( L^2(\T) \). Thus, we would get the desired bound
\[
\big\langle B_n,u_n(\cdot;g_n,\upsilon_n,z)\big\rangle = o_{\mathcal E}(1) 
\]
for  \(z\in K\subset \D \), $K$ is an arbitrary compact in $\D$, if we can only show that the functions \( u(\cdot;g,\upsilon_n,z) \) form a precompact family in \( L^2(\T) \) when \( z\in K \) and \( g\in \mathcal E \). Denote by \( \mathcal H \)  the Hilbert transform on $\T$. The functions $u(\xi;g_n,\upsilon_n,z)$ can be written as
\[
\big(1-\overline z \xi\big)^{-1} e^{-2\ic(\mathcal H g)(\xi)} e^{-2\ic(\mathcal H \log\upsilon_n)(\xi)}.
\]
Consider the following subsets of \( L^2(\T) \):
\[
\begin{cases}
S_1 & = \big\{\big(1-z\overline\xi\big)^{-1}:~ z\in K\big\}, \smallskip \\
S_2 & =\left\{ e^{2\ic(\mathcal H g)(\xi)}: g\in \mathcal{E}\right\}, \smallskip \\
S_3 & =\left\{e^{2\ic(\mathcal H \log\upsilon_n)(\xi)}: n\in \N\right\}.
\end{cases}
\]
The compactness of $S_1$ in $L^2(\T)$ is trivial. Since \( \mathcal E \) is compact in \( C(\T) \), it is also compact in \( L^2(\T) \). The operator \( \mathcal H \) is a bounded on \( L^2(\T) \). Hence, \( \mathcal H\mathcal E \) is a compact subset of real-valued functions in \( L^2(\T) \). Given two real numbers $f$ and $g$, we have
\begin{equation}
\label{2-2}
|e^{\ic f}-e^{\ic h}|^2=4\sin^2((f-h)/2) \leq 4|f-h|^p, \quad p\in(0,2].
\end{equation}
Therefore, $S_2$ is compact in \( L^2(\T) \). Finally, since \( \log\upsilon_n \to \log\upsilon \) in \( L^1 (\T) \) due to our condition \( (B_\T) \), it holds that \( \mathcal H\log\upsilon_n \to \mathcal H\log\upsilon \) in \( L^p (\T) \) for any \( p\in (0,1) \) by Kolmogorov's theorem. Using \eqref{2-2} with any such \( p \) yields that the functions \( \exp(2\ic\mathcal H \log\upsilon_n) \) converge to \( \exp(2\ic\mathcal H \log\upsilon ) \) in \( L^2(\T) \), which shows that $S_3$ is precompact in $L^2(\T)$. Since $S_1,S_2$, and $S_3$ are bounded in $L^\infty(\T)$, the product set $S_1S_2S_3$ is  precompact in $L^2(\T)$ as claimed. That completes the proof of the lemma.
\end{proof}

\subsection{Proof of Theorem~\ref{thm:2}}

To translate conditions \( (A_{[-1,1]})\)--\((D_{[-1,1]}) \) into conditions \( (A_\T)\)--\((D_\T) \), set \( b_{2n,i} = \phi(a_{n,i}) \), \( i\in\{1,2,\ldots,2n \} \). That is, \( J(b_{2n,i})=a_{n,i} \), where \( J(z) = ( z+z^{-1})/2\) is the Joukovski map. Then,
\[
\tau_n(z) = \prod_{a_{n,i}\neq\infty} \frac{(\zeta-b_{2n,i})(1-\zeta b_{2n,i})}{2a_{n,i}b_{2n,i}\zeta} = \frac{W_{2n}(\zeta)W_{2n}^*(\zeta)}{c_n\zeta^{2n}},
\]
\( z=J(\zeta) \), where \( c_n=\prod_{a_{n,i}\neq\infty}2a_{n,i}b_{2n,i} \), \( W_{2n}(\zeta) = \prod_{i=1}^{2n}(\zeta-b_{2n,i}) \), and we used the conjugate symmetry of the multi-set \( \{ b_{2n,1},b_{2n,2},\ldots,b_{2n,2n}\} \). Clearly, conditions \((D_{[-1,1]}) \) and \((D_\T)\) are equivalent to each other. Define \( \mathcal E=\{h\circ J:h\in\mathcal K\} \) and let \( g_{2n}=h_n\circ J \). It trivially holds that condition \( (C_{[-1,1]}) \) implies condition \( (C_\T) \). Every measure $\mu$ defined on $[-1,1]$ can be mapped to a measure $\sigma$ on $\T$ by the formula
\[
2\sigma(A)=\mu(J(A_+)) + \mu(J(A_-)),
\]
where $A$ is  any  subset in \( \T \), \( A_+=A\cap\{e^{\ic t}: t\in[0,\pi)\} \), and \( A_-=A\setminus A_+ \). For example, the mapping of the arcsine law $\omega_{[-1,1]}$ results in the normalized Lebesgue measure $m$ on the circle. More generally, if one has a function  \( v \) that is integrable with respect to \( \omega_{[-1,1]} \), the measure \(vd\omega_{[-1,1]}\) is mapped to  \(\upsilon dm$, where \( \upsilon= v\circ J \). We use this map to define measures \( \sigma_{2n} \) on \( \T \) that correspond to the measures $\tilde\mu_n$ on \( [-1,1] \).  This gives
\[
d\sigma_{2n} = \upsilon_{2n} dm + d\sigma_{2n}^s  \qandq d\sigma = \upsilon dm + d\sigma^s,
\]
where  \( \upsilon_{2n}=\tilde v_n \circ J$ and $\upsilon=v \circ  J \). Notice that condition \( (B_{[-1,1]}) \) implies  \( (B_\T) \). Similarly, condition \( (A_{[-1,1]}) \) implies condition \( (A_\T) \). 

Let now \( \phi_{2n} \) be the polynomials satisfying orthogonality relations \eqref{ortho-circle} with the above defined \( (\sigma_{2n},g_{2n}, W_{2n} ) \). By \cite[Theorem~11.5]{Szego} (see also \cite[Lemma~4.13]{St00}), it holds that
\[
p_n^2(z) = \frac{\phi_{2n}^*(\zeta)^2}{2c_n\zeta^{2n}}\frac{(1+\phi_{2n}(\zeta) / \phi_{2n}^*(\zeta))^2}{1+\phi_{2n}(0)/\phi_{2n}^*(0)} \qandq \gamma_n^2 = 2^{2n-1}\frac{\alpha_{2n}^2}{c_n}\left(1+\frac{\phi_{2n}(0)}{\phi_{2n}^*(0)}\right),
\]
where \( z=J(\zeta) \). Moreover, as observed in \cite[Lemma~4.3]{St00},  \( \widetilde G_n(z) = D_{2n}(\zeta) \). Thus, asymptotic formulae \eqref{A8-1} and \eqref{A8-2} follow from \eqref{A1-1} and \eqref{A1-2}, where one needs to use the identity
\[
\frac{\widetilde G_n^2(z)}{\tau_n(z)} \cdot \prod_{i=1}^{2n}\frac{\phi(z)-\phi(a_{n,i})}{1-\overline{\phi(a_{n,i})}\phi(z)}  = c_n\left(\frac{\zeta^nD_{2n}(\zeta)}{W_{2n}^*(\zeta)}\right)^2. 
\]

\section{OPs with Varying Weights and Szeg\H{o} Measures}
\label{sec:sm}

In this section we present the first generalization of Totik's theorem~\hyperref[thm:totik]{(Totik)}. To define the weight of orthogonality, we consider triples \( (\mu_n,h_n,\omega_n) \), where \( \mu_n,\omega_n \) are measures on a closed interval \( \Delta=[\alpha,\beta] \) and \( h_n \) is a continuous function on it. Assume further that
\emph{
\begin{itemize}
\item[$(A)$] there exists a finite  measure $\mu$ on $\Delta$ such that 
\[
\limsup_{n\to\infty} \int fd\mu_n \leq \int fd\mu
\]
for every nonnegative function \( f\in C(\Delta) \);
\item[$(B)$] if $v_n$ and $v$ denote the Radon-Nikodym derivatives of \( \mu_n \) and \( \mu \) with respect to the arcsine distribution \( \omega_\Delta \), see \eqref{szego}, then  \( \|\log v_n- \log v\|_{L^1(\omega_\Delta)}\to 0 \) as \( n\to\infty \); \smallskip
\item[$(C)$] the functions \( h_n \) belong to \( \mathcal K \), a fixed compact subset of \( C(\Delta) \); \smallskip
\item[$(D)$] \( d\omega_n(x)=\omega_n^\prime(x)dx \) are probability measures whose densities \( \omega_n^\prime \) form a uniformly equicontinuous family on every compact subset of \( (\alpha,\beta) \) and satisfy
\[
\begin{cases}
\omega_n^\prime(x)  & \gtrsim |w_\Delta(x)|^{\varkappa_L}, \quad x\in [\alpha+n^{-\tau},\beta-n^{-\tau}], \\
\omega_n^\prime(x) &  \lesssim |w_\Delta(x)|^{\varkappa_U}, \quad  x\in (\alpha,\beta),
\end{cases}
\]
for some \( \tau>0 \) and \( \varkappa_L,\varkappa_U>-2 \), see \eqref{phi-w}.
\end{itemize}}

We recall our notation:  given a nonnegative  measure \( \mu \) on an interval \( \Delta \) and a continuous nonnegative function \( f \) on \( \Delta \),  \( T_n(f\mu) \) denotes the \( n \)-th monic orthogonal polynomial with respect to the measure \( f\mu \). 

\begin{theorem}
\label{thm:3}
Assume that the triples \( (\mu_n,h_n,\omega_n) \) satisfy conditions \((A)\)--\((D)\) and define  \( \theta_n := 2nV^{\omega_n}+h_n \). Then, 
\begin{equation}\label{sapsad32}
T_n\left(e^{\theta_n}\mu_n\right)(z) = \left(1+o_{\mathcal K}(1)\right) \exp\left( n \int\log(z-x)d\omega_n(x)\right) \frac{G_n(e^{h_n}\mu_n;\infty)}{G_n(e^{h_n}\mu_n;z)}
\end{equation}
as \( n\to\infty \) locally uniformly in \( D_\Delta = \overline\C\setminus \Delta \). Moreover,
\[
\int_\Delta T_n^2\left(e^{\theta_n}\mu_n\right)(x)e^{\theta_n(x)}d\mu_n(x) = 2\big(1+o_{\mathcal K}(1)\big)G_n^2(e^{h_n}\mu_n;\infty).
\]
\end{theorem}

We prove Theorem~\ref{thm:3} in three steps that we organize as separate lemmas.

\begin{lemma}
\label{lem:vw3}
It is enough to prove Theorem~\ref{thm:3} for \( \Delta=[-1,1] \) only.
\end{lemma}
\begin{proof}
Let \( l(z) = az+b \) be any linear transformation with \( a>0 \) and \( b \) real. Set \( \Delta^{(l)}:=l^{-1}(\Delta) \) and, given a measure \( \mu \) on \( \Delta \), let \( \mu^{(l)} \) denote a  measure on \( \Delta^{(l)} \) such that \( \mu^{(l)}(B) = \mu(l(B)) \) for any  set \( B\subseteq \Delta^{(l)} \). Notice that \( \omega_\Delta^{(l)} = \omega_{\Delta^{(l)}} \) and that the Radon-Nikodym derivative of \( \mu^{(l)} \) with respect to the Lebesgue measure (resp. \( \omega_{\Delta^{(l)}} \)) is equal to \( a\mu^\prime(l(x)) \) (resp. \( v(l(x) \)), where \( \mu^\prime(x) \) (resp. \( v(x) \)) is the Radon-Nikodym derivative of \( \mu \) with respect to the Lebesgue measure (resp. \( \omega_\Delta \)). Observe also that
\[
V^\omega(l(x)) = -\int\log|l(x)-l(y)|d\omega^{(l)}(y) = -\log a + V^{\omega^{(l)}}(x)
\]
for any  measure \( \omega \). Set \( \theta_n^{(l)}:=2nV^{\omega_n^{(l)}}+h_n^{(l)}\),  \( h_n^{(l)} = h\circ l \). Hence, it holds that
\[
T_n(e^{\theta_n}\mu_n)(l(z)) = a^n  T_n\left(e^{\theta_n^{(l)}}\mu_n^{(l)}\right)(z).
\]
Since \( w_\Delta(l(z)) = aw_{\Delta^{(l)}}(z) \), the above considerations show that triples \( (\mu_n^{(l)},h_n^{(l)},\omega_n^{(l)}) \) satisfy conditions \( (A) \)--\( (D) \) on \( \Delta^{(l)} \) if the triples \( (\mu_n,h_n,\omega_n) \) satisfy  \( (A) \)--\( (D) \) on \( \Delta \). Finally, we get from \eqref{outer} and \eqref{SzegoFun} that
\[
e^{n \int\log(l(z)-x)d\omega_n(x)} \frac{G(e^{h_n}\mu_n;\infty)}{G(e^{h_n}\mu_n;l(z))} = a^ne^{n \int\log(z-x)d\omega_n^{(l)}(x)} \frac{G(e^{h_n^{(l)}}\mu_n^{(l)};\infty)}{G(e^{h_n^{(l)}}\mu_n^{(l)};z)},
\]
which finishes the proof of the lemma.
\end{proof}

Condition \( (D) \), placed on the measures \( \omega_n \) in Theorem~\ref{thm:3}, comes from \cite[Theorem~10.2]{Totik}. Under this assumption, it was shown there that there exist polynomials \( H_n \), \( \deg H_n \leq n \), that do not vanish on \( [-1,1] \) and such that the functions \( \iota_n = e^{2n V^{\omega_n}}|H_n|^2 \) satisfy
\begin{equation}
\label{iota}
\begin{cases}
 0< \iota_n(x)\leq 1, \quad x\in [-1,1], \smallskip \\
\displaystyle  \lim_{n\to\infty}\int\log\iota_n d\omega_{[-1,1]} = 0.
\end{cases}
\end{equation}
We remark that \cite[Theorem~10.2]{Totik} was formulated on \( [0,1] \), but its results can be easily brought to \( [-1,1] \) by a linear transformation. In that theorem,  we put \( \gamma=1/2 \) and \( u\equiv 1 \),  and the degree satisfies \( \deg H_n = n-i_n \), where \( i_n\to\infty \). The non-vanishing of $H_n$ was claimed only on \( (-1,1) \), but it is clear from the construction, see \cite[pages~58 and~75]{Totik}, that these polynomials also do not vanish at the endpoints. Set
\begin{equation}\label{sapsad29}
\tau_n(z) := \frac{H_n(z)\overline{H_n(\bar z)}}{|H_n(0)|^2}.
\end{equation}
The polynomial \( \tau_n \) has an even degree not exceeding \( 2n \) and has the following properties:
\begin{itemize}
\item \( \tau_n \) has real coefficients,  does not vanish on $[-1,1]$, and \( \tau_n(0)=1 \);
\item  if we denote  the zeros of \( \tau_n \) by \( a_{n,1},\ldots,a_{n,\deg\tau_n} \), then at least a half of them, see \cite[page~94]{Totik}, are located in $\{|\re\, z|<0.9, |\im \,z|>L_n/n\}$ with $\lim_{n\to\infty}L_n=+\infty$ (this guarantees that condition \( (D_{[-1,1]})\) of Theorem~\ref{thm:2} is satisfied by $\tau_n$).
\end{itemize}

\begin{lemma}
\label{lem:vw4}
Under the conditions of Theorem~\ref{thm:3} with \( \Delta=[-1,1] \), it holds that
\begin{multline}
\label{sapsad31}
T_n^2\left(e^{\theta_n}\mu_n\right)(z) = \big(1+o_{\mathcal K}(1)\big) \frac{G^2(e^{h_n}\iota_n\mu_n;\infty)}{G^2(e^{h_n}\iota_n \mu_n;z)} \times \\  \frac1{2^{2n}} \prod_{j=1}^{\deg \tau_n}(2a_{n,j}\phi(a_{n,j})) \frac{\tau_n(z)}{\phi^{2n-\deg \tau_n}(z)}  \prod_{j=1}^{\deg\tau_n}\frac{1-\overline{\phi(a_{n,j})}\phi(z)}{\phi(z)-\phi(a_{n,j})} 
\end{multline}
locally uniformly in \( D=\overline\C\setminus[-1,1] \), where, as before, \( \phi = \phi_{[-1,1]} \), see \eqref{phi-w}. Moreover,
\[
\int_{-1}^1 T_n^2\left(e^{\theta_n}\mu_n\right)e^{\theta_n}d\mu_n = \big(1+o_{\mathcal K}(1)\big) \frac{G^2(e^{h_n}\iota_n\mu_n;\infty)}{2^{2n-1}}\prod_{j=1}^{\deg \tau_n}(2\phi(a_{n,j})).
\]
\end{lemma}
\begin{proof}
Since monic orthogonal polynomials do not depend on the normalization of the measure of orthogonality, we have that
\[
T_n\left(e^{\theta_n}\mu_n\right) = T_n\left(e^{h_n}\tau_n^{-1}\tilde\mu_n\right), \quad \tilde\mu_n := \iota_n \mu_n. 
\]

Let us show that conditions \( (A_{[-1,1]}) \)--\( (D_{[-1,1]}) \) of Theorem~\ref{thm:2} are satisfied by the triples \( (\tilde\mu_n,h_n,\tau_n) \). We have already mentioned that the polynomials \( \tau_n \) fulfill \( (D_{[-1,1]}) \). Moreover, condition \( (C) \) of Theorem~\ref{thm:3} is identical to condition \( (C_{[-1,1]}) \) of Theorem~\ref{thm:2}.  Since the Radon-Nikodym derivative of \( \tilde\mu_n \) with respect to \( \omega_{[-1,1]} \) is \( \iota_n v_n \), the functions \( \iota_n \) obey the first line of \eqref{iota}, and
\[
\big|\log(\iota_nv_n)(x) - \log v(x) \big| \leq \big|\log v_n(x) - \log v(x) | - \log\iota_n(x),
\]
condition \( (B_{[-1,1]}) \) of Theorem~\ref{thm:2} follows from condition \( (B) \) of Theorem~\ref{thm:3} and the second line of \eqref{iota}. Finally, condition \( (A) \) of Theorem~\ref{thm:3} implies condition \( (A_{[-1,1]}) \) of Theorem~\ref{thm:2} for the same measure \( \mu \) due to the upper bound in the first line of \eqref{iota}.  

If $p_n(\sigma)(z)=\gamma_nz^n+\cdots$ denotes the $n$-th orthonormal polynomial with respect to some measure $\sigma$, then we can write
\[
T_n(\sigma)(z) = \gamma_n^{-1}p_n(\sigma)(z), \quad \text{where} \quad \gamma_n^{-2}=\int T_n^2(\sigma)d\sigma.
\]
So, the first claim of the lemma is deduced from Theorem~\ref{thm:2}. To get the second one, we first observe that \eqref{sapsad29} implies \( |H_n(0)|^2 = \prod_i a_{n,i} \). Then, 
\[
\int_{-1}^1 T_n^2\left(e^{\theta_n}\mu_n\right)e^{\theta_n}d\mu_n  = \frac1{|H_n(0)|^2} \int_{-1}^1 T_n^2\left(e^{h_n}\tau_n^{-1}\tilde\mu_n\right)e^{h_n}\frac{d\tilde\mu_n}{\tau_n}
\]
and we only need to apply \eqref{A8-2} to the last integral.
\end{proof}

\begin{lemma}
\label{lem:vw5}
Theorem~\ref{thm:3} holds on \( \Delta=[-1,1] \). 
\end{lemma}
\begin{proof}
We will show that the right-hand side of \eqref{sapsad31} can be written in a form consistent with \eqref{sapsad32}. We readily get from \eqref{outer} and \eqref{SzegoFun} that
\[
G^2(e^{h_n}\iota_n\mu_n;z)G_n^{-2}(e^{h_n}\mu_n;z) = \Omega(\iota_n;z) = \Omega\big(e^{2nV^{\omega_n}};z\big) \Omega\big(|H_n|^2;z\big),
\]
where \( \Omega(\cdot;\cdot) \) is used as a shorthand for \( \Omega_{[-1,1]}(\cdot;\cdot) \). Let us show that
\[
\begin{cases}
\Omega\big(|H_n|^2;z\big) & \displaystyle = \tau_n(z)\phi^{\deg \tau_n}(z) \left( \prod_{j=1}^{\deg \tau_n} a_{n,j}\right)\prod_{j=1}^{\deg \tau_n}\frac{1-\overline{\phi(a_{n,j})}\phi(z)}{\phi(z)-\phi(a_{n,j})}, \medskip \\
\Omega\big(e^{V^{\omega_n}};z\big) & \displaystyle = \phi^{-1}(z)  \exp\left( \int \log(z-x)d\omega_n(x)\right).
\end{cases}
\]
Both equalities follow from the same general principle: if \( f \) is a continuous function on \( [-1,1] \) and \( \Omega \) is a holomorphic non-vanishing function in \( D \) such that \( |\Omega| \) is continuous in the entire extended complex plane and \( |\Omega| = f \) on \( [-1,1] \), then \( \Omega = \Omega(f;\cdot) \). Continuity of \( |H_n|^2 \) is obvious while continuity of \( V^{\omega_n} \) follows from condition \( (D) \) and properties of the logarithmic potentials. Recall that \( \deg \tau_n \) is an even integer and that \( 2z\phi(z) \to 1 \) as \( z\to\infty \). Now, to complete the proof of the lemma, it  remains to notice that the above explicit representations yield
\[
\Omega\big(|H_n|^2;\infty\big) = \prod_{i=1}^{\deg \tau_n}(2\phi(a_{n,i}))^{-1} \qandq \Omega\big(e^{V^{\omega_n}};\infty\big) = 2. \qedhere
\]

\end{proof}

\section{OPs with Varying Weights and Uniformly Szeg\H{o} Measures}
\label{sec:usm}

In this section, the weights of orthogonality come from a single measure \( \mu \) (rather than a sequence of measures \( \mu_n \) as in Theorem~\ref{thm:3}) and triples \( (h_n,\kappa_n, \omega_n) \), where each \( h_n,\kappa_n \) are continuous functions on \( \Delta(\mu) \), the convex hull of the support of \( \mu \), and \( \omega_n \) is a measure whose support now can be a proper subset of \( \Delta(\mu) \).  The latter possibility necessitates strengthening of the notion of a Szeg\H{o} measure.  We shall say that a measure \( \mu \) is \emph{uniformly Szeg\H{o}} on a closed interval \( \Delta\subseteq\Delta(\mu) \), and denote this by \( \mu\in\mathrm{USz}(\Delta) \), if \( \mu\in\mathrm{Sz}(\Delta) \) and for any sequence of closed intervals \( \{\Delta_n\} \) such that \( \Delta_n\subseteq \Delta(\mu) \) and \( \Delta_n\to\Delta \) as \( n\to\infty \), there is $n_0$ such that \( \mu\in\mathrm{Sz}(\Delta_n) \) for $n\ge n_0$ and 
\begin{equation}
\label{UnSzego}
\lim_{n\to\infty}\int \big|\log \mu^\prime(x) - \log \mu^\prime(l_{\Delta\to\Delta_n}(x)) \big| d\omega_\Delta(x) \to 0,
\end{equation}
where \( l_{\Delta\to\Delta_n} \) is a linear function with a positive leading coefficient that maps \( \Delta \) onto \( \Delta_n \).  We adopt the term  ``uniformly Szeg\H{o}'' to emphasize that small perturbations of the endpoints of \( \Delta \) result in small changes of the value of the Szeg\H{o} function at infinity. In fact, this is true for the whole Szeg\H{o} function locally uniformly on \( \overline{\C}\backslash \Delta \).

\begin{proposition}
\label{prop:sa2}
Let \( \mu \) be a compactly supported measure such that \( \mu\in\mathrm{USz}(\Delta) \) for some \( \Delta\subseteq\Delta(\mu) \) and \( \{h_n\} \) be a sequence of continuous functions that converges uniformly on some closed interval that contains \( \Delta \) in its interior to a continuous function \( h \). Further, let \( \{\Delta_n\} \) be collections of closed intervals such that \( \Delta_n\subseteq \Delta(\mu) \) and \( \Delta_n\to\Delta \) as \( n\to\infty \). Then, the asymptotics
\begin{equation}
\label{SzegoCont}
G\big(e^{h_n}\mu_{|\Delta_n};z\big) =(1+o(1)) G\big(e^h\mu_{|\Delta};z\big).
\end{equation}
holds locally uniformly in $D_\Delta$
\end{proposition}

Since the concept of uniformly Szeg\H{o} measures is important to our analysis, we provide a different characterization of this class. To this end, given an integrable function \( \theta \) on an interval \( \Delta \), we let
\[
(I_\gamma\theta)(x) :=  \frac1{\sqrt\pi}\int_\gamma^x \frac{\theta(t)dt}{\sqrt{|x-t|}}, \quad x\in\Delta,
\]
where \( \gamma\in\Delta \) is fixed. That is a version of the so-called Riemann-Liouville fractional integral and it corresponds to the exponent \( 1/2 \). As an integral transform, \( I_\gamma \) is a continuous operator from \( L^1(\Delta) \) into weak-\( L^2 (\Delta) \), see \cite[Lemma~2.13]{MR4191495} (these are Lebesgue spaces with respect to the Lebesgue measure on \( \Delta \)). 
\begin{proposition}
\label{prop:sa3}
Let \( [\alpha,\beta]=\Delta \subseteq \Delta(\mu) \). The following are equivalent:
\begin{itemize}
\item[(i)] \( \mu\in\mathrm{USz}(\Delta) \);
\item[(ii)] \( (I_\gamma\log\mu^\prime)(x) \), \( x\in\Delta(\mu) \), is continuous at \( \beta \) and \( \alpha \), where \( \gamma \in(\alpha,\beta) \) is any;
\item[(iii)] for every \( \epsilon>0 \) there exists \( d_\epsilon>0 \) such that 
\[
\left| \int_{a_\alpha}^{b_\alpha}\frac{\log \mu^\prime(t) dt}{\sqrt{t-a_\alpha}} \right|,\left| \int_{a_\beta}^{b_\beta}\frac{\log \mu^\prime(t) dt}{\sqrt{b_\beta-t}} \right|<\epsilon
\]
when \( \dist(\alpha,\{a_\alpha,b_\alpha\}),\dist(\beta,\{a_\beta,b_\beta\})<d_\epsilon \).
\end{itemize}
\end{proposition}

If \( \theta\in L^p(\Delta) \) for some \( p>2 \), then \( I_\gamma\theta \)  is H\"older continuous on \( \Delta \) with exponent at least \( 2-1/p \), see \cite[Theorem~12]{MR1544927}. Since \( \mu^\prime \) is an integrable function, \( \log^+\mu^\prime \) is in \( L^p(\Delta) \) for any \( p>2 \). Hence, \( I_\gamma\log^+\mu^\prime \) is necessarily H\"older continuous and therefore Proposition~\ref{prop:sa3} could be equivalently stated with \( \log\mu^\prime \) replaced by either \( |\log\mu^\prime| \) or \( \log^-\mu^\prime \).

Below, we assume that the triples \( (h_n,\kappa_n,\omega_n) \) are such that
\emph{
\begin{itemize}
\item[$(1)$] \( \omega_n \) is a probability measure supported on some interval \( [\alpha_n,\beta_n] = \Delta_n\subseteq \Delta(\mu) \) and these intervals converge to an interval \( \Delta =[\alpha,\beta]\) as \( n\to\infty \);
\item[$(2)$] We have that \( d\omega_n(x)=\omega_n^\prime(x)dx \) and the densities \( \omega_n^\prime \)  form a uniformly equicontinuous family on any compact subset of $(\alpha,\beta)$. Moreover, they satisfy the bounds
\begin{equation}
\label{w_cond_1}
|w_{\Delta_n}(x)|^{\varkappa_L} \lesssim \omega_n^\prime(x) \lesssim|w_{\Delta_n}(x)|^{\varkappa_U}, \quad x\in(\alpha_n,\beta_n),
\end{equation}
for some \( \varkappa_L,\varkappa_U>-2 \). If \( \alpha_n>\alpha(\mu) \), we additionally assume that the inequalities
\begin{equation}
\label{w_cond_2}
\omega_n^\prime(x)\lesssim |x-\alpha_n|^{\widehat\varkappa_U}, \quad x\in(\alpha_n,\alpha_n+\delta),
\end{equation} hold for  $\widehat\varkappa_U>0$ and $n$-independent positive $\delta$. The similar assumption is made for the case when  \( \beta_n<\beta(\mu) \);
\item[$(3)$] The functions \( h_n \) belong to \( \mathcal K \), a fixed compact subset of \( C(\Delta(\mu)) \);
\item[$(4)$] The functions \( \kappa_n \) are such that
\[
\kappa_n(x)
\begin{cases}
\leq 0, & x\in \Delta(\mu), \\
= 0, & x\in \Delta_n,
\end{cases}
\]
and the bound
\begin{equation}
\label{w_cond_3}
|w_{\Delta_n}(x)|^\varkappa \lesssim |\kappa_n(x)| \lesssim |w_{\Delta_n}(x)|^\varkappa, \quad x\in\Delta(\mu)\setminus\Delta_n,
\end{equation}
holds for some \( \varkappa>0 \).
\end{itemize}
}
The functions $\kappa_n$ are introduced to model the behavior of the external field away from the support of the extremal measure. Recall \eqref{equilibrium condition}. It shows that
\[
w_n^{2n} = e^{2nV^{\omega_n}}e^{2n(\log w_n - V^{\omega_n})},
\]
where \( \omega_n \) is the extremal measure corresponding to the weight function \( w_n \). The difference \( \log w_n - V^{\omega_n}+\int \log w_nd\omega_n \) is non-positive on \( \Delta(\mu) \) and is identically zero on \( \supp(\omega_n) \). The functions \( \kappa_n \) account for those differences.

\begin{theorem} 
\label{thm:4}
Let \( \mu \) be a compactly supported measure. Assume that the triples \( (h_n,\kappa_n,\omega_n) \) possess the above properties (1)--(4) and that \( \mu\in\mathrm{USz}(\Delta) \). Set \( \theta_n := 2n(V^{\omega_n} + \kappa_n)+h_n \). Then, we have that
\[
T_n\left(e^{\theta_n}\mu\right)(z) = \left(1+o_{\mathcal K}(1)\right) \exp\left( n \int\log(z-x)d\omega_n(x)\right) \frac{G(e^{h_n}\mu_{|\Delta};\infty)}{G(e^{h_n}\mu_{|\Delta};z)}
\]
 as \( n\to\infty \) locally uniformly in \( D_\Delta \). Moreover,
\[
\int_{\Delta(\mu)} T_n^2\left(e^{\theta_n}\mu\right)(x)e^{\theta_n(x)}d\mu(x) = 2\big(1+o_{\mathcal K}(1)\big)G^2\big(e^{h_n}\mu_{|\Delta};\infty\big).
\]
In the above two formulae, the functions \( G(e^{h_n}\mu_{|\Delta};z) \) can be replaced by \( G(e^{h_n}\mu_{|\Delta_n};z) \).
\end{theorem} 

\subsection{Proof of Proposition~\ref{prop:sa2}}

Given a (real-valued) function \( u \) in \( L^1(\omega_\Delta) \), we set
\begin{equation}
\label{Dirichlet}
(H_\Delta u)(z) := \log|\Omega_\Delta(e^u;z)|, \quad z\in D_\Delta.
\end{equation}
Then, \( H_\Delta u \) is a harmonic function in \( D_\Delta \) whose non-tangential boundary values from above and below \( \Delta \) exist almost everywhere and are equal to \( u \), see \eqref{outer-b}. That is, \( H_\Delta u \) is a solution of the Dirichlet problem in \( D_\Delta \) with boundary data~\( u \). One can readily see from \eqref{outer} and \eqref{SzegoFun} that to prove the proposition it is enough to show that
\[
(H_{\Delta_n}u_n)(z) \to (H_\Delta u)(z) \qasq n\to\infty
\]
locally uniformly in \( D_\Delta \), where \( u_n =\log v_n +h_n\) and \( u=\log v+h \) while \( v_n \) and \( v \) are the Radon-Nikodym derivatives of \( \mu_{|\Delta_n} \) and \( \mu_{|\Delta} \) with respect to \( \omega_{\Delta_n} \) and \( \omega_\Delta \), respectively.  

Let \( l_n = l_{\Delta\to\Delta_n} \) be as in \eqref{UnSzego}. Observe that \( l_n(x) \) converges to \( x \) uniformly on \( \Delta \). Then,
\[
(H_{\Delta_n}u_n)(z) - (H_\Delta u)(z) = \int_\Delta \re\left(w_{\Delta_n}(z)  \frac{u_n(l_n(x))}{z-l_n(x)}- w_\Delta(z)\frac{u(x)}{z-x}\right) d\omega_\Delta(x).
\]
The function in parenthesis above can be rewritten as
\begin{multline*}
(w_{\Delta_n}(z)-w_\Delta(z))\frac{u(x)}{z-x} +w_{\Delta_n}(z)\frac{u_n(l_n(x))-u(x)}{z-l_n(x)} + \\   w_{\Delta_n}(z)\frac{(l_n(x)-x)u(x)}{(z-l_n(x))(z-x)} =: (J_{n,1}+J_{n,2}+J_{n,3})(x,z).
\end{multline*}
Observe that the functions \( w_{\Delta_n} - w_\Delta \) converge to zero uniformly in the whole extended complex plane. Hence,
\begin{equation}
\label{J123}
\int_\Delta J_{n,1}(x,z)d\omega_\Delta(x) \to 0   \qasq n\to\infty
\end{equation}
locally uniformly in \( D_\Delta \). Furthermore, since
\[
v_n(x) = \pi\mu^\prime(x)\sqrt{(x-\alpha_n)(\beta_n-x)}, \quad x\in\Delta_n=[\alpha_n,\beta_n],
\]
it follows that \( u_n(l_n(x))-u(x) \) is equal to
\[
\big[h_n( l_n(x)) - h(x) \big] + \frac12\log \frac{\beta_n-\alpha_n}{\beta-\alpha} +  \big[\log\mu^\prime(l_n(x)) - \log\mu^\prime(x) \big]
\]
on \( \Delta \), where \( \Delta=[\alpha,\beta] \). Due to uniform convergence of \( h_n \) to \( h \), uniform continuity of \( h \) on \( \Delta \), and \eqref{UnSzego}, these functions converge to zero in \( L^1(\omega_\Delta) \). As functions \( |w_{\Delta_n}(z)/(z-l_n(x))| \) are uniformly bounded for \( x\in\Delta \) and \( z \) on closed subsets of \( D_\Delta \), this necessarily yields that \eqref{J123} holds with  \( J_{n,1}(x,z) \) is replaced by \( J_{n,2}(x,z) \). Uniform convergence to zero of \( l_n(x)-x \)  on \( \Delta \) now guarantees that \eqref{J123} remains valid if \( J_{n,1}(x,z) \) is replaced by \( J_{n,3}(x,z) \) as well. This, of course, finishes the proof of the proposition.

\subsection{Proof of Proposition~\ref{prop:sa3}}

Set \( \theta(t) := |\log\mu^\prime(t)| \). As mentioned right after the statement of Proposition~\ref{prop:sa3}, we can consider \( (I_\gamma\theta)(x) \) instead of \( (I_\gamma\log\mu^\prime)(x) \).

\begin{lemma}
\label{lem:sa12}
Proposition~\ref{prop:sa3}(i) implies Proposition~\ref{prop:sa3}(ii).
\end{lemma}
\begin{proof}
We shall prove continuity at \( \beta \); the continuity at \( \alpha \) can be proven analogously. We need to show that
\[
\lim_{n\to\infty} (I_\gamma\theta)(\beta_n) = (I_\gamma\theta)(\beta) 
\]
for any sequence \( \{\beta_n\}\subset\Delta(\mu) \) such that \( \beta_n\to\beta \) as \( n\to\infty \). Clearly, in the limit above we can replace \( \gamma \) by \( \alpha \). Let \( l_n(t) = \alpha + \frac{\beta_n-\alpha}{\beta-\alpha}(t-\alpha) \). Observe also that
\[
 (I_\alpha\theta)(\beta_n) = \sqrt{\frac{\beta_n-\alpha}{\beta-\alpha}} \int_\alpha^\beta \frac{\theta(l_n(t))}{\sqrt{\beta-t}}dt.
\]
The claim of the lemma now follows from \eqref{UnSzego} and the estimate
\begin{align*}
\left| \int_\alpha^\beta \frac{\theta(t)}{\sqrt{\beta-t}}dt - \int_\alpha^\beta \frac{\theta(l_n(t))}{\sqrt{\beta-t}}dt \right| & \leq \pi\sqrt{\beta-\alpha}\int_\alpha^\beta |\theta(t) - \theta(l_n(t))|d\omega_\Delta(t) \\
& \leq \pi\sqrt{\beta-\alpha}\int_\alpha^\beta |\log\mu^\prime(t) - \log\mu^\prime(l_n(t))|d\omega_\Delta(t). \qedhere
\end{align*}
\end{proof}

\begin{lemma}
\label{lem:sa13}
Proposition~\ref{prop:sa3}(iii) implies Proposition~\ref{prop:sa3}(i).
\end{lemma}
\begin{proof}
Let \( \{ \Delta_n \}\) be a sequence of closed subintervals of \( \Delta(\mu) \) that converges to \( \Delta \). Pick \( \epsilon> 0 \) and let \( \delta\)  be some positive number to be specified later.  Then,
\begin{eqnarray}\nonumber
\int_\alpha^\beta  |\log\mu^\prime(t) - \log\mu^\prime(l_n(t))|d\omega_\Delta(t) \leq  \int_{\alpha+\delta}^{\beta-\delta}  |\log\mu^\prime(t) - \log\mu^\prime(l_n(t))|d\omega_\Delta(t) + \\ \int_{\beta-\delta}^\beta \big( \theta(t) + \theta(l_n(t)) \big) d\omega_\Delta(t) + \int_{\alpha}^{\alpha+\delta} \big( \theta(t) + \theta(l_n(t)) \big) d\omega_\Delta(t),\label{sapsad6}
\end{eqnarray}
where \( l_n(t) \) is the linear transformation with the positive leading coefficient that maps \( \Delta \) onto \( \Delta_n \). 
If $\alpha_n$ and $\beta_n$ are the endpoints of $\Delta_n$, i.e.,  \( \Delta_n = [\alpha_n,\beta_n] \), then
\[
\int_{\beta-\delta}^\beta \theta(l_n(t)) d\omega_\Delta(t) = \int_{\beta_n-\delta_n}^{\beta_n} \theta(t) d\omega_{\Delta_n}(t),
\] 
where \( \delta_n = \frac{\beta_n-\alpha_n}{\beta-\alpha}\delta \), and a similar equality holds for the integral of \( \theta(l_n(t)) \) on \( [\alpha,\alpha+\delta] \). Notice that $\lim_{n\to\infty}\delta_n=\delta$. Now, it becomes clear that the assumption (iii) of Proposition~\ref{prop:sa3} implies that there exists \( \delta>0 \) and $N\in \mathbb{N}$ such that
\begin{eqnarray*}
\int_{\beta-\delta}^\beta \theta(t) d\omega_\Delta(t)< \frac\epsilon5,\,\int_{\alpha}^{\alpha+\delta}\theta(t) d\omega_\Delta(t) < \frac\epsilon5,\,\hspace{3cm}\\
\int_{\beta-\delta}^\beta \theta(l_n(t)) d\omega_\Delta(t)< \frac\epsilon5,\,\int_{\alpha}^{\alpha+\delta}\theta(l_n(t)) d\omega_\Delta(t) < \frac\epsilon5
\end{eqnarray*}
for all $n\ge N$.
To see that the first integral in \eqref{sapsad6} also can be made smaller than \( \epsilon/5 \) for all  large enough \( n \), observe that \( d\omega_\Delta(t) \leq (\pi\delta)^{-1}dt \) on the interval of integration and that
\[
\lim_{n\to\infty}
\int_{\alpha+\delta}^{\beta-\delta}  |\log\mu^\prime(t) - \log\mu^\prime(l_n(t))|dt=0\,.
\]
Indeed, this limit is obviously zero if we replace $\log\mu'$ by a continuous function. The general case follows by  approximating  \( \log\mu^\prime \) in \( L^1(\Delta(\mu)) \)-norm with continuous functions.
\end{proof}

\begin{lemma}
\label{lem:sa14}
Proposition~\ref{prop:sa3}(ii) implies Proposition~\ref{prop:sa3}(iii).
\end{lemma}
\begin{proof}
Since \( (I_\gamma\theta)(x) \) is continuous at \( \beta \), there exists an interval, say \( [a,b] \), that contains \( \beta \) in its interior (unless \( \beta \) is the right endpoint of \( \Delta(\mu)\), in which case \( b=\beta \)), on which \( (I_\gamma\theta)(x) \) is bounded.  It is known \cite[Theorem~2.1]{SamkoKilbasMarichev} that
\begin{equation}
\label{I-0}
\theta(x) = \frac{d}{dx} \left(\frac1{\sqrt\pi} \int_\gamma^x \frac{(I_\gamma\theta)(t)}{\sqrt{x-t}}dt\right) = \frac{d}{dx} (I_\gamma(I_\gamma\theta))(x), \quad x\in(\gamma,b).
\end{equation}
Let us write \( (I_\gamma^2\theta)(x) \) for \( (I_\gamma(I_\gamma\theta))(x) \), which is an absolutely continuous function on \( [\gamma,b] \) that vanishes at \( \gamma \), see again \cite[Theorem~2.1]{SamkoKilbasMarichev}. Fix some \( \delta \in (0,(b-a)/2) \). Notice that
\begin{multline*}
\frac{d}{dx}\left( \int_\gamma^{x-\delta}\frac{(I_\gamma^2\theta)(t)}{\sqrt{x-t}}dt \right) = \frac{d}{dx}\left( \int^{x-\gamma}_\delta \frac{(I_\gamma^2\theta)(x-s)}{\sqrt s}d s \right) = \\
 \lim_{h\to 0}\int^{x-\gamma}_\delta  \frac{(I_\gamma^2\theta)(x+h-s)-(I_\gamma^2\theta)(x-s)}{h}\frac{ds}{\sqrt s} +\lim_{h\to 0} \frac1h\int_\gamma^{\gamma+h}\frac{(I_\gamma^2\theta)(t)}{\sqrt{x+h-t}}dt.
 \end{multline*}
The second limit is equal to zero due to continuity of the integrand and vanishing of \( (I_\gamma^2\theta)(t) \) at \( \gamma \). It is known,  see \cite[Theorem~6.9]{RoydenFitzpatrick}, that absolute continuity of a function is equivalent to uniform integrability of its divided differences. As  \( (I_\gamma^2\theta)(x) \) is absolutely continuous and \( 1/\sqrt{s} \) is continuous on \( [\delta,b-\gamma] \), we get from Vitali's convergence theorem and \eqref{I-0} that the first limit is equal to 
\[
\int_\delta^{x-\gamma} \frac{\theta(x-s)}{\sqrt s}ds 
\]
and hence
\begin{equation}
\label{I-1}
\frac{d}{dx}\left( \int_\gamma^{x-\delta}\frac{(I_\gamma^2\theta)(t)}{\sqrt{x-t}}dt \right) =\int_\delta^{x-\gamma} \frac{\theta(x-s)}{\sqrt s}ds = \int_\gamma^{x-\delta} \frac{\theta(t)}{\sqrt{x-t}}dt.
\end{equation}
Writing the outer \( I_\gamma \) transform explicitly and  changing the order of integration gives us
\begin{multline}
\label{I-2}
\int_\gamma^{x-\delta}\frac{(I_\gamma^2\theta)(t)}{\sqrt{x-t}}dt = \frac1{\sqrt\pi}\int_\gamma^{x-\delta} \left( \int_s^{x-\delta} \frac{dt}{\sqrt{(t-s)(x-t)}} \right) (I_\gamma\theta)(s)ds  = \\ \int_\gamma^{x-\delta}  F\left(1-\frac\delta{x-s} \right) (I_\gamma\theta)(s)ds, \quad F(s) := \frac1{\sqrt \pi} \int_0^s \frac{dt}{\sqrt{t(1-t)}}.
\end{multline}
 Again, we need to justify changing the order of differentiation and integration. To this end, assume now that \( x\in(a+2\delta,b) \). By the mean-value theorem and its very definition, the derivative of the last integral in \eqref{I-2} is equal to the limit as \( h\to0 \) of the following sum of three terms
 \begin{multline}
 \label{I-3}
 \delta\int_\gamma^a F^\prime\left(1-\frac\delta{x+\xi_h-s} \right) \frac{(I_\gamma\theta)(s)}{(x-s)^2}ds + \\
 \frac1h\int_a^{x-\delta} \left( F\left(1-\frac\delta{x+h-s}\right) - F\left(1-\frac\delta{x-s}\right) \right)  (I_\gamma\theta)(s)ds + \\ 
\frac1h\int_{x-\delta}^{x+h-\delta}F\left(1-\frac\delta{x+h-s}\right)  (I_\gamma\theta)(s)ds,
 \end{multline}
where \( \xi_h = \xi_h(x,s) \) is such that \( |\xi_h|\leq |h| \). The first term in the sum above converges to 
\[
 \sqrt{\frac\delta\pi} \int_\gamma^a \frac{(I_\gamma\theta)(s)ds}{(x-s)\sqrt{x-\delta-s}},
\]
when \( h\to0 \).  This follows from the dominated convergence theorem because \( (I_\gamma\theta)(s) \) is a fixed integrable function and the other factor in the integrand is a function continuous in  $s$ that converges uniformly when $h\to 0$. The second term in \eqref{I-3} has the following limit 
\[
\delta \int_a^{x-\delta} F^\prime\left(1-\frac\delta{x-s} \right) \frac{(I_\gamma\theta)(s)}{(x-s)^2}ds = \sqrt{\frac\delta\pi} \int_a^{x-\delta}\frac{(I_\gamma\theta)(s)ds}{(x-s)\sqrt{x-\delta-s}}
\]
as \( h\to0 \) according to Vitali's convergence theorem. Indeed,  \( (I_\gamma\theta)(s) \) is bounded on the interval of integration by assumptions of Proposition~\ref{prop:sa3}(ii) and the divided differences of $F(1-\delta(x-s)^{-1})$ are uniformly integrable. The last term in \eqref{I-3} can be rewritten as
\[
\frac\delta h \int_0^{h/(h+\delta)}  (I_\gamma\theta)\left(x+h-\frac\delta{1-t} \right)\frac{F(t)}{(1-t)^2}dt.
\]
Its limit as \( h\to0 \) is equal to \( 0 \) due to the boundedness of \( (I_\gamma\theta)(x) \) on \( (a,b) \) as well as the continuity of the function \( F(t)/(1-t)^2 \) around the origin and its vanishing  at $t=0$.  Altogether, we get from \eqref{I-1}, \eqref{I-2}, and the reasoning above that
\[
\int_\gamma^{x-\delta} \frac{\theta(t)dt}{\sqrt{x-t}} = \sqrt{\frac\delta\pi} \int_\gamma^{x-\delta} \frac{(I_\gamma\theta)(s)ds}{(x-s)\sqrt{x-\delta-s}} = \frac2{\sqrt\pi} \int_0^{ L_{x,\delta}} (I_\gamma\theta)\big(x-\delta-\delta t^2\big)\frac{dt}{1+t^2},
\]
where \( L_{x,\delta} = \sqrt{(x-\delta-\gamma)/\delta} \). We can use the identity $\int_0^\infty (t^2+1)^{-1}dt=\pi/2$ to rewrite 
\begin{multline*}
\frac1{\sqrt\pi}\int_{x-\delta}^x \frac{\theta(t)dt}{\sqrt{x-t}} = (I_\gamma\theta)(x) -\frac1{\sqrt\pi} \int_\gamma^{x-\delta} \frac{\theta(t)dt}{\sqrt{x-t}} = \frac2\pi \left( \int_{ \ell_\delta/\sqrt\delta}^\infty \frac{(I_\gamma\theta)(x)dt}{1+t^2} \right. \\   \left. - \int_{\ell_\delta/\sqrt\delta}^{ L_{x,\delta}} \frac{(I_\gamma\theta)(x-\delta-\delta t^2)dt}{1+t^2} + \int_0^{\ell_\delta/\sqrt\delta} \frac{(I_\gamma\theta)(x) - (I_\gamma\theta)(x-\delta-\delta t^2)}{1+t^2}dt\right)
\end{multline*}
for any positive \( \ell_\delta \). If we choose $\ell_\delta=\delta^{\frac 18}$, we get that
\[
\left| \int_{\ell_\delta/\sqrt\delta}^\infty \frac{(I_\gamma\theta)(x)dt}{1+t^2} \right| \leq |(I_\gamma\theta)(x)|\frac{\sqrt\delta}{\ell_\delta}=|(I_\gamma\theta)(x)|\delta^{\frac 38}\,.
\]
Similarly,
\begin{multline*}
\left| \int_{\ell_\delta/\sqrt\delta}^{ L_{x,\delta}} \frac{(I_\gamma\theta)(x-\delta-\delta t^2)dt}{1+t^2} \right| \leq \frac{\sqrt\delta}{\ell_\delta^2}\int_\gamma^{x-\delta-\ell_\delta^2}\frac{|(I_\gamma\theta)(s)|ds}{\sqrt{x-\delta-s}} \\
 \leq \frac{\sqrt\delta}{\ell_\delta^3} \int_\gamma^b |(I_\gamma\theta)(s)|ds=\delta^{\frac 18} \int_\gamma^b |(I_\gamma\theta)(s)|ds\lesssim\delta^{\frac 18}  \,.
\end{multline*}
Next, we have that
\[
\left|\int_0^{\ell_\delta/\sqrt{\delta}} \frac{(I_\gamma\theta)(x) - (I_\gamma\theta)(x-\delta-\delta t^2)}{1+t^2}dt \right| \leq \frac\pi2 \max_{s\in[x-\delta-\delta^{\frac 14},x-\delta]} \big|(I_\gamma\theta)(x) - (I_\gamma\theta)(s) \big|.
\]
Now, using continuity of \( (I_\gamma\theta)(x) \) at \( \beta \), the above estimates show that given \( \epsilon> 0\), we can always find \( d_\epsilon \) so that \( |x-\beta|<d_\epsilon\) and \(|x-\delta-\beta|<d_\epsilon \) imply
\[
\int_{x-\delta}^x \frac{\theta(t)dt}{\sqrt{x-t}} < \epsilon.
\]
This is precisely the statement of the second part in Proposition~\ref{prop:sa3}(iii). The first one is proved similarly.
\end{proof}

\noindent{\bf An example.} The uniform Szeg\H{o} condition is subtle and depends on the direction in which the point is approached, as shown by the following example. Given \( \epsilon\in[0,1] \), let \( \mu_\epsilon \) be an absolutely continuous measure on \( [-1,1] \) such that \( \log\mu^\prime_\epsilon(x) = - \theta_\epsilon(x) \), where
\[
\theta_\epsilon(x) := \begin{cases}
0, & x\in[-1,0], \\
x^{-1/2}(1 - \log x)^{-\epsilon}, & x\in(0,1].
\end{cases}
\]
When \( \epsilon=0 \), it holds that \( \mu_0\in\mathrm{Sz}([-1,a]) \) for any \( a\in(-1,1] \). Indeed, the claim is obvious when \( a\leq 0 \). When \( a>0 \), it holds that
\[
\frac1\pi\int_{-1}^a \frac{\theta_0(x)dx}{\sqrt{(x+1)(a-x)}} = \frac1\pi \int_0^a \frac{dx}{\sqrt{(x+1)x(a-x)}} \leq \frac1\pi \int_0^a \frac{dx}{\sqrt{x(a-x)}} = 1.
\]
This computation also shows that
\[
\frac1\pi \int_{-1}^\delta \frac{\theta_0(x)dx}{\sqrt{\delta-x}} = \begin{cases} 0, & \delta\leq0, \\ 1, & \delta>0, \end{cases}
\]
and therefore \( \mu_0\not\in\mathrm{USz}([-1,0]) \) as follows from Proposition~\ref{prop:sa3}(ii). On the other hand, when \( \epsilon>0 \), it holds that
\[
0\leq \frac1\pi \int_{-1}^\delta \frac{\theta_\epsilon(x)dx}{\sqrt{\delta-x}} \leq \frac1{(1-\log\delta)^\epsilon} \to 0 \qasq \delta\to0^+,
\]
and therefore \( \mu_\epsilon\in\mathrm{USz}([-1,0]) \) again by Proposition~\ref{prop:sa3}(ii) (in fact, \( \mu_\epsilon\in\mathrm{USz}([-1,a]) \)  for any \( a\in(-1,1] \) in this case). However, when \( \epsilon\leq 1 \), we have that \( \mu_\epsilon\not\in\mathrm{Sz}([0,1]) \) since
\[
\frac1\pi \int_0^1  \frac{\theta_\epsilon(x)dx}{\sqrt{x(1-x)}} \geq \frac1\pi \int_0^1 \frac{dx}{x(1-\log x)^\epsilon} = \frac1\pi\int_1^\infty \frac{du}{u^\epsilon} = \infty.
\]

\subsection{Proof of Theorem~\ref{thm:4}}

We prove Theorem~\ref{thm:4} in four steps organized as separate lemmas.  We write \( \Delta(\mu) = [\alpha(\mu),\beta(\mu)] \), \( \Delta_n = [\alpha_n,\beta_n] \), and \( \Delta = [\alpha,\beta] \).  

Let \( \varkappa \) be a parameter in \eqref{w_cond_3} and \( \{\delta_n\} \) be a sequence of positive numbers given by
\begin{equation}
\label{delta_n}
\delta_n = (\xi_n/n)^{2/(\varkappa+1)},
\end{equation}
where \( \xi_n\to 0 \) as \( n\to\infty \) (this sequence will be specified  later in the formula \eqref{ser20}). Set
\[
\Delta_n^* := [\alpha_n^*,\beta_n^*], \quad \alpha_n^* := \max\big\{\alpha_n-\delta_n,\alpha(\mu) \big\}, \quad \beta_n^* := \min\big\{\beta_n+\delta_n,\beta(\mu) \big\}.
\]
Our strategy consists in first applying Theorem~\ref{thm:3}  to obtain the asymptotics of polynomials $P_n^*$ orthogonal with respect to the same measures as \( T_n \) from Theorem~\ref{thm:4} but restricted to  $\Delta_n^*$. Then, we show that asymptotics of $T_n$ coincides with that of $P_n^*$.

\begin{lemma}
\label{lem:vw6}
Recall that \( \theta_n = 2n(V^{\omega_n} + \kappa_n)+h_n \). It holds that
\[
P_n^*(z) := T_n\left(e^{\theta_n}\mu_{|\Delta_n^*}\right)(z) = (1+o_{\mathcal K}(1)) F_n^*(z)
\]
locally uniformly in \( D_\Delta  \), where
\begin{equation}\label{ser21}
F_n^*(z) := \exp\left( n \int\log(z-x)d\omega_n(x)\right) \frac{G(e^{h_n}\mu_{|\Delta_n^*};\infty)}{G(e^{h_n}\mu_{|\Delta_n^*};z)}.
\end{equation}
\end{lemma}
\begin{proof}
To apply  Theorem~\ref{thm:3}, we need to rescale intervals \( \Delta_n^* \) to the fixed limit interval \( \Delta \). To this end, let \( l_n(x) = a_nx+b_n \), \( a_n>0 \), be the linear function that maps \( \Delta \) onto \( \Delta_n^* \). Since $\Delta_n^*\to\Delta$, we have  \( a_n\to1 \) and \( b_n\to0 \) as \( n\to\infty \). In the notation of Lemma~\ref{lem:vw3}, we denote the rescaled equilibrium measure by \( \tilde \omega_n := \omega_n^{(l_n)} \). Then, one has
\[
\supp(\tilde\omega_n) = l_n^{-1}(\Delta_n)\subseteq \Delta \qandq d\tilde\omega_n(x)  = \tilde\omega_n^\prime(x)dx = a_n\omega_n^\prime(l_n(x))dx.
\]
Write \( l_n^{-1}(\Delta_n)=[\tilde\alpha_n,\tilde\beta_n]\). Since \( a_n\to 1 \) as \( n\to\infty \), the upper bound in \eqref{w_cond_1} gives 
\[
\tilde\omega_n^\prime(x) \lesssim (x-\tilde\alpha_n)^{\varkappa_U/2}(\tilde\beta_n-x)^{\varkappa_U/2},  \quad x\in \big( \tilde\alpha_n,\tilde\beta_n \big).
\]
If \( \tilde\alpha_n>\alpha \) for some index \( n \), then \( \alpha_n>\alpha_n^*\geq \alpha(\mu) \) and \( (x-\tilde\alpha_n)^{\varkappa_U/2} \) can be replaced by \( (x-\tilde\alpha_n)^{\widehat\varkappa_U} \) in the above estimate as required by \eqref{w_cond_2}, where \( \widehat\varkappa_U>0 \). 
\begin{figure}[!ht]
\includegraphics[scale=.65]{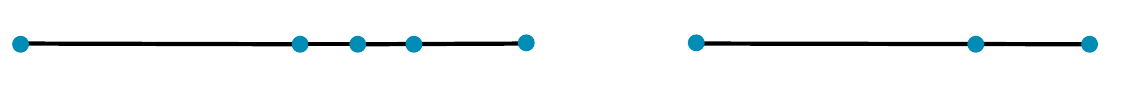}
\begin{picture}(0,0)
\put(-358,2){\(\alpha(\mu)=\alpha_n=\alpha_n^*=\alpha\)}
\put(-265,19){\(\beta\)}
\put(-248,19){\(\beta_n\)}
\put(-231,19){\(\beta_n^*\)}
\put(-200,19){\(\beta(\mu)\)}
\put(-45,0){\(\beta=l_n^{-1}(\beta_n^*)\)}
\put(-80,19){\(\tilde\beta_n=l_n^{-1}(\beta_n)\)}
\put(-160,2){\(\alpha(\mu)=\tilde\alpha_n\)}
\end{picture}
\caption{\small Segments \( \Delta\subset\Delta_n\subset\Delta_n^*\subset\Delta(\mu) \) and \( l_n^{-1}(\Delta_n)\subset\Delta=l_n^{-1}(\Delta_n^*) \).}
\label{fig:1}
\end{figure}
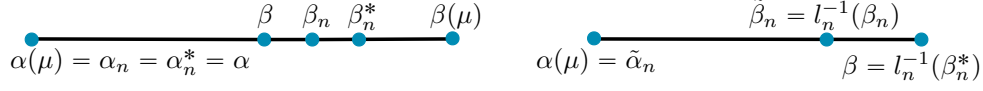
Similarly, if \( \tilde\beta_n<\beta \) for some \( n \), then \( (\tilde\beta_n-x)^{\varkappa_U/2} \) can again be replaced by \( (\tilde\beta_n - x )^{\widehat\varkappa_U} \). Either way, the upper bound in assumption \( (D) \) of Theorem~\ref{thm:3} is fulfilled. Similarly to the upper bound, we have that
\[
\tilde\omega_n^\prime(x) \gtrsim |w_{[\tilde\alpha_n,\tilde\beta_n]}(x)|^{\varkappa_L},  \quad x\in ( \tilde\alpha_n,\tilde\beta_n),
\]
by the lower bound in \eqref{w_cond_1}. If \( l_n^{-1}(\Delta_n) =\Delta \) for all \( n \), the above inequality gives the desired lower bound in assumption \( (D) \) of Theorem~\ref{thm:3}. If at least one of the intervals \( l_n^{-1}(\Delta_n) \) is a proper subinterval of \( \Delta \), then the corresponding upper bound requires that \( \varkappa_L>0 \). It can be readily checked that
\[
|w_{[\tilde\alpha_n,\tilde\beta_n]}(x)| \gtrsim  |w_\Delta(x)|, \quad x\in \big(\alpha+2(\tilde\alpha_n-\alpha), \beta-2(\beta-\tilde\beta_n)\big).
\]
Notice that \( a_n( \tilde\alpha_n-\alpha) = \alpha_n-\alpha_n^* \leq \delta_n \) and similarly that \( a_n(\beta-\tilde\beta_n) \leq \delta_n \). Then,   \eqref{delta_n} guarantees existence of \( \tau \geq \frac2{1+\varkappa} \)  such that
\[
\tilde\omega_n^\prime(x) \gtrsim  |w_\Delta(x)|^{\varkappa_L}, \quad x\in(\alpha+n^{-\tau},\beta-n^{-\tau}).
\] 
As equicontinuity of the functions \( \tilde\omega_n^\prime(x) \) on compact subsets of \( (\alpha,\beta) \) follows from the equicontinuity of the densities \( \omega_n^\prime(x) \), all the requirements of assumption \( (D) \) of Theorem~\ref{thm:3} are satisfied.

Next, let \( \tilde h_n = h_n\circ l_n \). Functions \( l_n \) converge to the identity map locally uniformly, hence assumption $(C)$ of Theorem~\ref{thm:3} follows from  assumption (3) of Theorem~\ref{thm:4} with the compact subset of \( C(\Delta) \) being the closure of \( \cup_n\{h\circ l_n : h\in\mathcal K\} \). 

Now, we verify assumption $(A)$ for the same measure \( \mu \) as in the statement of Theorem~\ref{thm:4} and a sequence of measures
\[
d\tilde \mu_n(x) = e^{2n\tilde\kappa_n(x)}d(\mu_{|\Delta_n^*})^{(l_n)}(x), \quad \tilde \kappa_n = \kappa_n\circ l_n.
\]
Given a continuous nonnegative function \( f \) on \( \Delta \), it holds that
\[
\int_\Delta fd(\mu_{|\Delta_n^*}^{(l_n)}-\mu) = \int_{\Delta\cap \Delta_n^*}( f\circ l_n^{-1}-f)d\mu + \int_{\Delta_n^*\setminus \Delta}  f\circ l_n^{-1} d\mu - \int_{\Delta\setminus\Delta_n^*} fd\mu.
\]
The first integral on the right-hand side above converges to zero due to uniform continuity of \( f \). The second one converges to zero because \( f \) is bounded and \( \cap_{n}\Delta_n^*\setminus \Delta=\emptyset\) so that \(\mu(\Delta_n^*\setminus \Delta)\to 0\) as \( n\to\infty \). The third integral is always nonnegative. Hence, as \( \tilde\kappa_n(x) \leq 0 \), assumption $(A)$ of Theorem~\ref{thm:3} is clearly fulfilled. 

To verify assumption $(B)$, we write
\[
\log\tilde v_n(x) = 2n\tilde \kappa_n(x) + \log v_{\Delta_n^*}(l_n(x)),
\]
where \( \tilde v_n \) is the Radon-Nikodym derivative of \( \tilde\mu_n \) with respect to \( \omega_\Delta \), see \eqref{szego}. Observe that
\[
|\log v_\Delta(x) - \log v_{\Delta_n^*}(l_n(x))| = |\log\mu^\prime(x) - \log\mu^\prime(l_n(x)) - \log a_n|.
\]
Since \( a_n\to 1\) and \( \mu\in\mathrm{Usz}(\Delta) \), \( \| \log v_\Delta - \log v_{\Delta_n^*}(l_n)\|_{L^1(\omega_\Delta)} \to 0 \) as \( n\to\infty \) due to assumption \eqref{UnSzego}. Furthermore, we have 
\[
\int_\Delta |\tilde\kappa_n(x)|d\omega_\Delta(x) = \int_{\Delta_n^*} |\kappa_n(x)| d\omega_{\Delta_n^*}(x)  \lesssim \int_{\Delta_n^*\setminus\Delta_n} |w_{\Delta_n}(x)|^\varkappa d\omega_{\Delta_n^*}(x),
\]
where we used \eqref{w_cond_3} and identity \( \kappa_n\equiv0 \) on \( \Delta_n \). Hence, it readily follows from the definition of \( \Delta_n^* \) that
\begin{equation}
\label{L1kappa}
2n\int_\Delta |\tilde\kappa_n(x)|d\omega_\Delta(x) \lesssim_\Delta n \int_0^{\delta_n} x^{(\varkappa+1)/2} d\omega_{[0,\delta_n]}(x) \lesssim_\Delta n \delta_n^{(\varkappa+1)/2} = \xi_n.
\end{equation}
Since $\lim_{n\to\infty}\xi_n=0$ by our assumptions, we get \( \| 2n\tilde \kappa_n \|_{L^1(\omega_\Delta)} \to 0 \) as \( n\to\infty \), which shows that condition $(B)$ of Theorem~\ref{thm:3} is also satisfied. 

Altogether, we have that the triples \( (\tilde\mu_n,\tilde h_n,\tilde\omega_n) \) satisfy all the conditions of Theorem~\ref{thm:3} and therefore
\[
T_n\left(e^{2nV^{\tilde\omega_n}+\tilde h_n}\tilde\mu_n\right)(z) = \left(1+o_{\mathcal K}(1)\right) \exp\left( n \int\log(z-x)d\tilde \omega_n(x)\right) \frac{G(e^{\tilde h_n}\tilde \mu_n;\infty)}{G(e^{\tilde h_n}\tilde\mu_n;z)}
\]
locally uniformly in \( D_\Delta  \). Computations in Lemma~\ref{lem:vw3} now show that the above formula is equivalent to the statement of the lemma.
\end{proof}

The next lemma provides a simple uniform estimate on the boundary behavior for the sequence of outer functions. Recall the definition of the function $G$ given in \eqref{SzegoFun} and \eqref{outer}.
\begin{lemma}
\label{lem:vw7}
There exists a non-decreasing function \( \epsilon_\mu(t) \) such that \( \lim_{t\downarrow 0}\epsilon_\mu(t)=0 \) and
\[
\log G^{-2}\big(e^{h_n}\mu_{|\Delta_n};\beta_n+t\big) \leq \frac{\epsilon_\mu(t)}{\sqrt t}
\] 
for every \( t\in[0,\beta(\mu)-\beta_n] \) and all \( n \). Moreover, an analogous estimate holds with \( \beta_n+t \) replaced by \( \alpha_n-t \) for \( t\in[0,\alpha_n-\alpha(\mu)] \).
\end{lemma}
\begin{proof}
From \eqref{outer} and \eqref{SzegoFun}, we get
\[
\log G^{-2}\big(e^{h_n}\mu_{|\Delta_n};\beta_n+t\big) \leq w_{\Delta_n}(\beta_n+t) \int_{\Delta_n} \frac{|h_n(x)|+|\log v_{\Delta_n}(x)|}{\beta_n+t-x}d\omega_{\Delta_n}(x).
\]
The application of the Cauchy integral formula gives
\[
\frac1{w_{\Delta_n}(z)} = \frac1{\pi\ic}\int_{\Delta_n} \frac1{x-z}\frac{dx}{w_{\Delta_n+}(x)} = \int_{\Delta_n} \frac{d\omega_{\Delta_n}(x)}{z-x},
\]
where we used the fact that \( w_{\Delta_n+} = \ic |w_{\Delta_n}| \) on \( \Delta_n \). Since \( \mathcal K \) is a compact family,  the functions \( |h_n| \) are uniformly bounded. Therefore, we get that
\[
w_{\Delta_n}(\beta_n+t) \int_{\Delta_n} \frac{|h_n(x)|}{\beta_n+t-x}d\omega_{\Delta_n}(x) \lesssim_{\mathcal K} 1.
\]
Recall further that \(  v_{\Delta_n}(x) = \pi\mu^\prime(x)|w_{\Delta_n}(x)|\). We have 
\begin{align}
w_{\Delta_n}(\beta_n+t) \int_{\Delta_n} \frac{|\log|w_{\Delta_n}(x)||}{\beta_n+t-x}d\omega_{\Delta_n}(x) & \lesssim_\Delta  1 + \sqrt t \int_{\beta_n-1}^{\beta_n} \frac{-\log(\beta_n-x)}{\beta_n+t-x} \frac{dx}{\sqrt{\beta_n-x}} \\
& \lesssim_\Delta 1 + |\log t|
\end{align}
when \( t\in[0,\beta(\mu)-\beta_n] \), where the first estimate is deduced exactly as the estimate for the integrals of \( |h_n| \) above and the second one follows by making a substitution \( \beta_n-x=ty\). Moreover,
\begin{align}
\int_{\Delta_n} \frac{|\log\mu^\prime(x)|}{\beta_n+t-x}d\omega_{\Delta_n}(x)  &\lesssim_{\Delta} \frac{1}{\sqrt t} \int_{\alpha_n}^{\beta_n-\sqrt t} \frac{|\log\mu^\prime(x)|}{\sqrt{\beta_n-x}}dx+\frac1t \int_{\beta_n-\sqrt t}^{\beta_n} \frac{|\log\mu^\prime(x)|}{\sqrt{\beta_n-x}}dx\\
& \stackrel{ \eqref{UnSzego}}{\lesssim_\Delta}  \frac{\|\log\mu^\prime\|_{L^1(\omega_\Delta)}}{\sqrt t} + \frac1t \int_{\beta_n-\sqrt t}^{\beta_n} \frac{|\log\mu^\prime(x)|}{\sqrt{\beta_n-x}}dx\,.
 \end{align}
We define $\epsilon(t)$ by the formula:
\[
\epsilon(t) := \sup_{n>N_0} \int_{\beta_n-\sqrt t}^{\beta_n} \frac{|\log\mu^\prime(x)|}{\sqrt{\beta_n-x}}dx,
\]
where $N_0$ is sufficiently large to make sure that integrals converge. Clearly, this is a non-decreasing function of \( t \). We claim that \( \lim_{t\downarrow 0} \epsilon(t) =0 \). Indeed, assume to the contrary that there exist \( \epsilon_0>0 \), a sequence \( \{t_m\} \) decreasing to \( 0 \), and a set \( \{m_n\} \) such that
\begin{equation}
\label{contrapos}
\int_{\beta_{m_n}-\sqrt t_m}^{\beta_{m_n}} \frac{|\log\mu^\prime(x)|}{\sqrt{\beta_{m_n}-x}}dx \geq \epsilon_0.
\end{equation}
 If $\{m_n\}$ is bounded, it contains a constant sequence. However, \eqref{contrapos} cannot hold along this sequence as the integrals of a fixed integrable function over sets of decreasing measure must vanish. On the other hand, if \( \{m_n\} \) contains a strictly increasing sequence, \eqref{contrapos} contradicts Proposition~\ref{prop:sa3}(iii) since \( \mu\in \textrm{USz}(\Delta) \) (recall that we can replace \( \log\mu^\prime \) by \( |\log\mu^\prime|\) in that proposition). Hence, altogether, 
\[
w_{\Delta_n}(\beta_n+t) \int_{\Delta_n}\frac{|\log\mu^\prime(x)|}{\beta_n+t-x}d\omega_{\Delta_n}(x) \lesssim_{\mu,\Delta} 1 + \frac{\epsilon(t)}{\sqrt t}.
\]
Combining the previous estimates, we obtain
\[
\log G^{-2}\big(e^{h_n}\mu_{|\Delta_n},\beta_n+t\big) \lesssim_{\Delta,\mu,\mathcal K} \frac{\sqrt t(1 + |\log t|) + \epsilon(t)}{\sqrt t}.
\]
Since the numerator on the right-hand side is estimated by a non-decreasing function that has zero limit at zero, the claim of the lemma follows.
\end{proof}

In the proof of the next lemma, we will be using the following estimate. Suppose \( f \in H^2(\D) \), then  the Cauchy integral formula and Cauchy-Schwarz inequality  give
\[
|f(z)| \leq \int_\T \frac{|f(\eta)|}{|\eta-z|}\frac{|d\eta|}{2\pi} \leq \frac{ \|f\|_{L^2(\T)} }{\sqrt{1-|z|}}.
\]
Thus, if \( F \in H^2(\overline\C\setminus\Delta) \), then \( F\circ \phi_\Delta \) is a function in \( H^2(\D) \), see \eqref{phi-w}, and therefore
\begin{equation}
\label{HardyEst}
|F(\beta+\delta)| \leq \frac{\|F\|_{L^2(\omega_\Delta)}}{\sqrt{1-\phi_\Delta(\beta+\delta)}} \lesssim_\Delta \frac{\|F\|_{L^2(\omega_\Delta)}}{\sqrt[4]\delta}.
\end{equation}

\begin{lemma}
\label{lem:vw8}
The sequence \( \{\xi_n\} \) from \eqref{delta_n} can be chosen so that
\[
\int_{\Delta(\mu)\setminus\Delta_n^*} P_n^*(x)^2e^{\theta_n(x)}d\mu(x) \to 0 \qasq n\to\infty.
\]
\end{lemma}
\begin{proof}
As in the proof of Lemma~\ref{lem:vw6}, let  \( l_n(x)=a_nx+b_n \), \( a_n>0 \), be the linear function that maps of $\Delta$ onto $\Delta_n^*$. Due to the compactness of \( \mathcal K \), the functions \( h_n-h_n\circ l_n \) converge to zero uniformly on \( \Delta \). Thus, according to Proposition~\ref{prop:sa2}, 
\begin{equation}
\label{limitGs}
\begin{cases}
G\big(e^{h_n}\mu_{|\Delta_n^*};z\big) & = \big(1+o_{\mathcal K}(1)\big)G\big(e^{h_n}\mu_{|\Delta};z\big), \\ 
G\big(e^{h_n}\mu_{|\Delta_n};z\big) & = \big(1+o_{\mathcal K}(1)\big)G\big(e^{h_n}\mu_{|\Delta};z\big), 
\end{cases}
\end{equation}
locally uniformly in \( D_\Delta \).  Hence, the last asymptotic formula of Theorem~\ref{thm:3} pulled back to the intervals \( \Delta_n^* \) as in Lemma~\ref{lem:vw6} shows 
\begin{equation}
\label{kvhlk}
G^{-2}(e^{h_n}\mu_{|\Delta_n};\infty)\int_{\Delta_n^*} P_n^*(x)^2e^{\theta_n(x)}d\mu(x) = 2+o_{\mathcal K}(1).
\end{equation}
Define
\[
F_n(z) = \exp\left( n \int\log(z-x)d\omega_n(x)\right) \frac{G(e^{h_n}\mu_{|\Delta_n};\infty)}{G(e^{h_n}\mu_{|\Delta_n};z)}
\]
(where we restricted \( \mu \) to \( \Delta_n \) instead of \( \Delta_n^* \) as in Lemma~\ref{lem:vw6}).
Using identity \( \kappa_n(x)\equiv 0 \) on \( \Delta_n \), we get
\[
|F_{n\pm}(x)|^{-2} = G^{-2}(e^{h_n}\mu_{|\Delta_n};\infty)e^{\theta_n(x)}v_{\Delta_n}(x)
\]
almost everywhere on \( \Delta_n \) by \eqref{outer-b}, \eqref{szego} and \eqref{SzegoFun}.
 Therefore, 
\[
2+o_{\mathcal K}(1) \geq \int_{\Delta_n} \big|P_n^*(x)/F_{n\pm}(x)\big|^2 d\omega_{\Delta_n}(x),
\]
where we first reduce the interval of integration from \( \Delta_n^* \) to \( \Delta_n \) in \eqref{kvhlk} and then drop the singular part of \( \mu \). Since the potentials \( V^{\omega_n} \) are continuous in \( \C \),  each function \( P_n^*/F_n \) is a product of a Szeg\H{o} function and a bounded analytic function. As such, it belongs to \( H^2(D_{\Delta_n}) \). Hence,  \eqref{HardyEst} gives us
\[
\big|(P_n^*/F_n)(x)\big|^2 \lesssim_{\Delta,\mathcal K} 1/\sqrt{x-\beta_n}, \quad x\in[\beta_n,\beta(\mu)].
\]
Observe also that \( G(e^{h_n}\mu_{|\Delta_n};\infty) \lesssim_{\mu,\mathcal K} 1 \) by the compactness of \( \mathcal K \) and \eqref{limitGs}.

Now, we are ready to estimate the integrals in the statement of the lemma. We only carry out the estimates on \( [\beta_n^*,\beta(\mu)] \) as the estimates on \( [\alpha(\mu),\alpha_n^*] \) can be done analogously. Assume that \( \beta_n^*<\beta(\mu) \) (otherwise we have nothing to prove). In this case \( \beta_n^*=\beta_n+\delta_n \). Using the  bounds we just obtained, Lemma~\ref{lem:vw7}, and \eqref{w_cond_3}, we get
\begin{multline*}
\int_{\beta_n^*}^{\beta(\mu)} P_n^*(x)^2e^{\theta_n(x)}d\mu(x) = \int_{\beta_n^*}^{\beta(\mu)} \left|\frac{P_n^*(x)}{F_n(x)}\right|^2 \frac{G^2(e^{h_n}\mu_{|\Delta_n};\infty)}{G^2(e^{h_n}\mu_{|\Delta_n};x)} e^{2n\kappa_n(x)+h_n(x)} d\mu(x) \\ \lesssim_{\Delta,\mu,\mathcal K} \int_{\beta_n^*}^{\beta(\mu)} \exp\left( -C_\Delta n(x-\beta_n)^{\varkappa/2} + \frac{\epsilon_\mu(x-\beta_n)}{\sqrt{x-\beta_n}} \right) \frac{d\mu(x)}{\sqrt{x-\beta_n}}.
\end{multline*}
Observe that we can always make \( \epsilon_\mu(t) \) decay as slowly as we need (this will not affect the inequality in Lemma~\ref{lem:vw7}). In particular, we can redefine it if necessary to ensure
\[
\frac{\epsilon_\mu(t)}{\sqrt{t}} - \frac12\log t \leq 2\frac{\epsilon_\mu(t)}{\sqrt{t}} \quad \Leftrightarrow \quad -\frac12\sqrt t\log t \leq \epsilon_\mu(t).
\]
Thus, by taking \( t=x-\beta_n \), we get
\begin{equation}
\label{ser11}
\int_{\beta_n^*}^{\beta(\mu)} P_n^*(x)^2e^{\theta_n(x)}d\mu(x)  \lesssim_{\Delta,\mu,\mathcal K} \max_{\delta_n\leq t \leq T}\exp\left( C \left( -n t^{\varkappa/2} + \frac{\epsilon_\mu(t)}{\sqrt{t}} \right) \right),
\end{equation}
where \( T := \sup_n (\beta(\mu)-\beta_n)>0 \). Let 
\begin{equation} 
T_n := (2\epsilon_\mu(T)/n )^{2/(\varkappa+1)}.\label{ser2a} 
\end{equation}
First, consider the range
\begin{equation}
\label{ser2}
 T_n\leq t \leq T.
 \end{equation}
Since  \( \epsilon_\mu(t) \) is non-decreasing,  we get
\begin{equation}
\label{ser1}
\epsilon_\mu(t) \le \epsilon_\mu(T) \stackrel{\eqref{ser2a}}{=} (n/2)T_n^{(\varkappa+1)/2} \stackrel{\eqref{ser2}}{\leq} (n/2)t^{(\varkappa+1)/2}.
\end{equation}
For  \( t \) in that range, we can also write 
\[
- n t^{\varkappa/2} + \epsilon_\mu(t)/\sqrt{t} \stackrel{\eqref{ser1}}{\leq} - (n/2) t^{\varkappa/2} \stackrel{\eqref{ser2}}{\leq} - (n/2) T_n^{\varkappa/2} \stackrel{\eqref{ser2a}}{=} - \big(n\epsilon_\mu^\varkappa(T)/2\big)^{1/(\varkappa+1)}.
\]
Clearly, the right-hand side above converges to \( - \infty \) as $n\to\infty$. Next, consider $t$ that satisfy
\begin{equation} 
\label{ser7}
\delta_n\leq t \leq T_n.
\end{equation}
Since \( \epsilon_\mu(T_n)\to 0 \), we can choose \( \xi_n \) in \eqref{delta_n} so that 
\begin{equation}\label{ser3}
2\epsilon_\mu(T_n)\leq \xi_n.
\end{equation}
It again follows from the monotonicity of $\epsilon_\mu(t)$ that
\begin{equation}
\label{ser6}
\epsilon_\mu(t) \leq \epsilon_\mu(T_n) \stackrel{\eqref{ser3}}{\leq} (1/2)\xi_n \stackrel{\eqref{delta_n}}{=} (n/2)\delta_n^{(\varkappa+1)/2} \stackrel{\eqref{ser7}}{\leq}  (n/2)t^{(\varkappa+1)/2}.
\end{equation}
Hence, for \( t \) satisfying \eqref{ser7}, we obtain 
\[
-n t^{\varkappa/2} + \epsilon_\mu(t)/\sqrt{t} \stackrel{\eqref{ser6}}{\leq} - (n/2) t^{\varkappa/2} \stackrel{\eqref{ser7}}{\leq}  - (n/2) \delta_n^{\varkappa/2} \stackrel{\eqref{delta_n}}{=} -(1/2)(n\xi_n^\varkappa)^{1/(\varkappa+1)}. 
\]
To show that the right-hand side of \eqref{ser11} converges to zero when $n\to\infty$ and finish the proof of the lemma, it only remains to choose  $\xi_n$ so that
\begin{equation}
\lim_{n\to\infty}\xi_n=0, \quad \lim_{n\to\infty}n\xi_n^\varkappa=+ \infty, \qandq 2\epsilon_\mu(T_n) \leq \xi_n. \qedhere
\label{ser20}
\end{equation}
\end{proof}

\begin{lemma}
\label{lem:vw9}
Theorem~\ref{thm:4} holds.
\end{lemma}
\begin{proof}
Let \( F_n^*(z) \) be as in \eqref{ser21} and
\[
\Omega_{\kappa_n}(z) := \Omega_{\Delta_n^*}\big(e^{-n\kappa_n};z),
\]
which is an outer function in \( H^2(D_{\Delta_n^*}) \). Its traces satisfy \( |\Omega_{\kappa_n\pm}|^2 = e^{-2n\kappa_n} \) almost everywhere on \( \Delta_n^* \), see \eqref{outer-b}. Thus, 
\[
|(\Omega_{\kappa_n}F_n^*)_\pm(x)|^{-2} = G_n^{-2}e^{\theta_n(x)}v_{\Delta_n^*}(x)
\]
almost everywhere on \( \Delta_n^* \) due to \eqref{outer-b}, \eqref{szego} and \eqref{SzegoFun}, where \( G_n := G(e^{h_n}\mu_{|\Delta_n^*};\infty) \). For brevity, put \( P_n = T_n(e^{\theta_n}\mu) \). Recall that \( F_n^*(z)/z^n\to 1 \) and \( w(z)/z\to 1 \) as \( z\to\infty \),
\[
F_{n-}^*(x) = \overline{F_{n+}^*(x)} \qandq w_{\Delta_n^*+}(x) = - w_{\Delta_n^*-}(x) = \ic|w_{\Delta_n^*}(x)|, \quad x\in\Delta_n^*.
\]
Then, using \eqref{arcsine} and the Cauchy integral formula, we write
\begin{align*}
\frac{P_n(z)-P_n^*(z)}{ (\Omega_{\kappa_n}F_n^*)(z)w_{\Delta_n^*}(z)} &= \frac1{2\pi\ic}\oint_{\Gamma_n} \frac{P_n(s)-P_n^*(s)}{z-s} \frac{1}{(\Omega_{\kappa_n}F_n^*)(s)} \frac{ds}{w_{\Delta_n^*}(s)} \\
& = \int_{\Delta_n^*} \frac{P_n(x)-P_n^*(x)}{z-x} \re\left( \frac1{(\Omega_{\kappa_n}F_n^*)_+(x)} \right) d\omega_{\Delta_n^*}(x),
\end{align*}
where \( \Gamma_n \) is any counter-clockwise oriented Jordan curve that separates \( \Delta_n^* \) and \( z \). Since the absolute value of the real part of a complex number does not exceed its absolute value, the application of the Cauchy-Schwarz inequality yields
\begin{align}\nonumber
\left|\frac{P_n(z)-P_n^*(z)}{ (\Omega_{\kappa_n}F_n^*)(z)}\right|^2 & \leq G_n^{-2}\frac{|w_{\Delta_n^*}(z)|^2}{\dist(z,\Delta_n^*)^2} \int_{\Delta_n^*} (P_n-P_n^*)^2e^{\theta_n} v_{\Delta_n^*}d\omega_{\Delta_n^*} \\
\label{ser22}
& \leq G_n^{-2}\frac{|w_{\Delta_n^*}(z)|^2}{\dist(z,\Delta_n^*)^2}  \int_{\Delta_n^*} (P_n-P_n^*)^2 e^{\theta_n}d\mu.
\end{align}
As was observed in the previous lemma, it follows from the compactness of \( \mathcal K \) and \eqref{limitGs} that the sequence \( G_n^{-2} \) is uniformly bounded. Because \( P_n^* \) is the \( n \)-th monic orthogonal polynomials with respect to \( e^{\theta_n}\mu_{|\Delta_n^*} \) and \( P_n \) is a monic polynomials of degree \( n \), we have that
\[
\int_{\Delta_n^*} P_nP_n^*e^{\theta_n}d\mu = \int_{\Delta_n^*} (P_n^*)^2 e^{\theta_n}d\mu,
\]
from which we  deduce 
\begin{equation}
\label{ser23a}
\int_{\Delta_n^*} (P_n-P_n^*)^2 e^{\theta_n}d\mu = \int_{\Delta_n^*} P_n^2e^{\theta_n}d\mu - \int_{\Delta_n^*} (P_n^*)^2e^{\theta_n}d\mu.
\end{equation}
Recall that the \( n\)-th monic orthogonal polynomial $P_n$ satisfies a bound
\begin{equation}
\label{ser23}
\int_{\Delta(\mu)} P_n^2 e^{\theta_n}d\mu \leq \int_{\Delta(\mu)} P^2e^{\theta_n}d\mu
\end{equation}
for any monic polynomial \( P \) of degree \( n \). Take $P=P_n^*$. Then,  \eqref{ser23a}, simple majorization, and \eqref{ser23} give us
\begin{align*}
\int_{\Delta_n^*} (P_n-P_n^*)^2 e^{\theta_n}d\mu & \leq \int_{\Delta(\mu)} P_n^2e^{\theta_n}d\mu - \int_{\Delta_n^*} (P_n^*)^2e^{\theta_n}d\mu \\
& \stackrel{\eqref{ser23}}{\leq}  \int_{\Delta(\mu)} (P_n^*)^2e^{\theta_n}d\mu - \int_{\Delta_n^*} (P_n^*)^2e^{\theta_n}d\mu \\
& = \int_{\Delta(\mu)\setminus\Delta_n^*} (P_n^*)^2e^{\theta_n}d\mu \to  0 \qasq n\to\infty,
\end{align*}
where the last conclusion is  Lemma~\ref{lem:vw8}.  From \eqref{outer}, \eqref{delta_n}, and \eqref{L1kappa}, we get
 \( \Omega_{\kappa_n}(z) = 1+o(1) \) locally uniformly in the complement of \( \Delta \). Since the intervals \( \Delta_n^* \) converge to \( \Delta \),  \eqref{ser22}, the above estimates, and Lemma~\ref{lem:vw6} ensure that
\[
P_n(z) = P_n^*(z) + o_{\mathcal K}(1) (\Omega_{\kappa_n} F_n^*)(z) = \big( 1+ o_{\mathcal K}(1) \big) F_n^*(z)
\]
locally uniformly in \( D_\Delta \). The first claim of Theorem~\ref{thm:4} now follows from \eqref{limitGs}. It also holds that
\begin{align}
\int_{\Delta_n^*} (P_n^*)^2e^{\theta_n}d\mu \leq \int_{\Delta_n^*} P_n^2e^{\theta_n}d\mu \leq \int_{\Delta(\mu)} P_n^2e^{\theta_n}d\mu
\end{align}
since \( P_n^* \) is the \( n \)-th monic orthogonal polynomial with respect \( e^{\theta_n}d\mu_{|\Delta_n^*} \) and similarly
\begin{align}
\int_{\Delta(\mu)} P_n^2e^{\theta_n}d\mu  \leq \int_{\Delta(\mu)}(P_n^*)^2e^{\theta_n}d\mu = \int_{\Delta_n^*} (P_n^*)^2e^{\theta_n}d\mu + o(1),
\end{align}
where we used  Lemma~\ref{lem:vw8}  for the last equality. The second claim of the theorem is now a consequence of \eqref{kvhlk}, the second line of \eqref{limitGs}, and the uniform boundedness of \( G^{-2}(e^{h_n}\mu_{|\Delta};\infty) \).
\end{proof}

\section{OPs with Varying Weights and Strongly Szeg\H{o} Measures}
\label{sec:ssm}

Just like in the previous two sections,  we are interested in the strong asymptotics of the monic orthogonal polynomials \( T_n(w_N^{2N}\mu_n) \). We again take \( \{\mu_n\} \) to be a collection of Szeg\H{o} measures on an interval \( \Delta \), with densities satisfying some mild assumptions and \( \{ w_N \} \) to be a sequence of smooth weights. The difference is that  we now work in the following asymptotic regime:
\begin{equation}
\label{nN}
N=N(n), \quad n/N \to 0 \qasq n\to\infty.
\end{equation}
We think of the measures \( \mu_n \) as slowly changing with \( n \) while the varying behavior comes primarily from the weights \( w_N^{2N} \). 

Recall property \eqref{equilibrium condition} of the weighted extremal measures. In this setting, we let
\[
 \omega_{n,N} := \omega_{w_N^{N/n}} \qandq \omega_n:=\omega_{n,N(n)}.
\]
In this section we place the following assumptions on the weights \( w_N \):
\emph{
\begin{itemize}
\item[(i)] Each \( w_N \) is a three times continuously differentiable  function on \( \Delta \)\footnote{Including the one-sided derivatives at \( \alpha,\beta\).} satisfying
\begin{equation}
 w_N^\prime(x)>0, \,\, x<\eta_N, \qandq	w_N^\prime(x)<0,\,\, x>\eta_N, 
 \end{equation}
for some \(\eta_N\in\Delta\) (if $\eta_N$ is an endpoint of $\Delta$, one of the above conditions is vacuous). We assume \( w_N(\eta_N) = 1 \) and \( w_N\geq\delta \) on \( \Delta \) for some \( \delta>0 \) independent of \( N \);
\item[(ii)] There exists a three times continuously differentiable function \( w \) on \( \Delta \) satisfying
\begin{equation}
 w^\prime(x)>0, \,\, x<\eta, \qandq w^\prime(x)<0,\,\, x>\eta,
 \end{equation}
for some point $\eta\in \Delta$ (again, if $\eta$ is an endpoint of $\Delta$, one of the above conditions is vacuous). We assume  \((w^\prime/w)^\prime(\eta)<0 \) and, in addition, it holds that \( w_N^{(k)} \to w^{(k)} \) uniformly on \( \Delta \) as \( N\to\infty \) for each \( k\in\{0,1,2,3\} \).
\end{itemize}
}

As in Section~\ref{sec:usm}, only assuming Szeg\H{o} condition on a single interval is not sufficient for our purposes. Recall \eqref{arcsine} and \eqref{szego}. Let $\gamma>0$ be fixed. We say that \( \mu \) is  \emph{strongly Szeg\H{o} with exponent \( \gamma \)} on a closed interval \( I \) if 
\begin{equation}
\label{cond_szego}
\int_{I^\prime} |\log v_{I^\prime}(x)|d\omega_{I^\prime}(x) \lesssim |I^\prime|^{-\gamma}
\end{equation}
for any closed subinterval \( I^\prime \subseteq I \) (in particular, \( \mu\in\mathrm{Sz}(I^\prime)\) for each such \( I^\prime \)), where the constant in the above estimate is independent of \( I^\prime \). 

\begin{theorem}
\label{thm:5}
Let \( \{w_N \} \) be a sequence of weights on \( \Delta \) satisfying conditions $(i)$ and $(ii)$. Let \( \{\mu_n\} \) be a sequence of measures on \( \Delta \) that are strongly Szeg\H{o} with the same exponent \( \gamma\in(\frac12,1) \) on a closed interval \( I_\eta\subset\Delta \) that contains \(\eta \) in its interior if \( \eta \) is an interior point of \( \Delta \) (or as an endpoint if $\eta\in \partial\Delta$). Assume further that masses \( |\mu_n| \) as well as constants of proportionality in \eqref{cond_szego} are bounded uniformly in \( n \). Then, given \eqref{nN}, the asymptotics
\begin{equation}
\label{totik}
T_n(w_N^{2N}\mu_n)(z) = (1+o(1)) \exp\left( n\int \log(z-x)d\omega_n(x)\right) \qasq n\to\infty
\end{equation}
holds locally uniformly in \( \overline\C\setminus\{\eta\} \).
\end{theorem}

The rest of this section is devoted to the proof of Theorem~\ref{thm:5}. In the first subsection below, we discuss implications of the conditions $(i)$ and $(ii)$ for the weighted extremal measures \( \omega_{n,N} \).  The second subsection contains  auxiliary lemmas. The proof of Theorem~\ref{thm:5} is given in the final subsection.

\subsection{Weighted Extremal Measures}

Throughout the subsection, we assume that conditions $(i)$ and $(ii)$ of Theorem~\ref{thm:5} are satisfied. These conditions immediately imply that \( w\geq \delta \) on \( \Delta \) and \( w(\eta)=1 \). Furthermore, it is also clear that each \( w_N \) reaches its maximum on \( \Delta \) at \( \eta_N \), \(\eta_N\to \eta\), and that we can shrink \( I_\eta \) if necessary so that
\begin{equation}
\label{sd1}
\begin{cases}
L_N^\prime(x)<0, & x\in I_\eta, \\
L^\prime(x)<0, & x\in I_\eta,
\end{cases}
\quad \text{where} \quad
\begin{cases}
L_N(x) &:= (w_N^\prime/w_N)(x), \\
L(x)  &:= (w^\prime/w)(x).
\end{cases}
\end{equation}
We remark that if \( w=\exp(-V^\sigma) \) for some  measure $\sigma$ whose support is disjoint from \( \Delta \), then
 \[
L^\prime(x) = - \int\frac{d\sigma(s)}{(x-s)^2}
 \]
and the inequalities in \eqref{sd1} hold on the entire interval \( \Delta \). 

The assumptions $(i)$ and $(ii)$ have potential-theoretic implications for the extremal measures \( \omega_{n,N} \) and we discuss them below.

\begin{lemma}
\label{lem:4.2}
For all \( n/N \) small enough, the measure \( \omega_{n,N} \) is absolutely continuous and supported on an interval which we denote by \( \Delta_{n,N} \). Moreover,  \( \eta_N\in\Delta_{n,N} \) and \( |\Delta_{n,N}| =o(1) \) as \( n/N \to 0 \).
\end{lemma}
\begin{proof}
Let \( \nu_{n,N} := (n/N)\omega_{n,N} \). It is known, \cite[Theorem~I.1.3]{SaffTotik} that
\begin{equation}
\label{4.2.1}
V^{\nu_{n,N}}(x) - \log w_N(x) 
\begin{cases}
\geq \ell_{n,N}, & \text{q.e. on }~\Delta, \\
\leq \ell_{n,N}, & x\in\supp(\nu_{n,N}),
\end{cases}
\end{equation}
for some constant \( \ell_{n,N} \), see \eqref{equilibrium condition}. Let \( K\subset \Delta\setminus\{\eta\}\) be a compact set. Then, it follows from the  definition of a logarithmic potential and assumptions $(i)$ and $(ii)$ that
\begin{align}
V^{\nu_{n,N}}(x) - \log w_N(x) & \geq \frac nN \log\frac1{|\Delta|} - \log w_N(x) \nonumber \\
& \geq \frac nN \log\frac1{|\Delta|} + 2\varepsilon_K \geq \varepsilon_K, \quad x\in K,
\label{4.2.2}
\end{align}
for some \( \varepsilon_K>0 \) and all \( n/N \) small enough. On the other hand, given an interval \( I\subset \Delta \), it follows from the top inequality in \eqref{4.2.1} that
\begin{align}
\ell_{n,N} \leq \int \big(V^{\nu_{n,N}} - \log w_N\big)d\omega_I  \leq \frac nN \log \frac 4{|I|} - \log\min_{I}w_N,
\end{align}
where we used Fubini-Tonelli's theorem to interchange the order of integration and an estimate \( V^{\omega_I}(z) \leq \log(4/|I|) \) for all \( z\in \C \) well-known in potential theory. Since $w(\eta)=1$, we conclude  that for any \( \varepsilon>0 \) one can choose a small enough interval containing \( \eta \) so that
\begin{equation}
\label{4.2.3}
\ell_{n,N} \leq \varepsilon
\end{equation}
for all \( n/N \) small enough. Inequalities \eqref{4.2.2} and \eqref{4.2.3} as well as the bottom inequality in \eqref{4.2.1} imply that \[ \supp(\omega_{n,N}) \cap K = \varnothing \] for all \( n/N \) small enough. Our claim that \( \eta_N\in \supp(\omega_{n,N}) \) follows, for instance, from \cite[Theorem~IV.1.3]{SaffTotik} if we take  $w_N^{\ell N/n}\cdot 1, \ell\ge 0$ as a required weighted polynomial to test there.

It follows from the first part of the proof and \eqref{sd1} that
\begin{equation}
\label{4.2.4}
\supp(\omega_{n,N}) \subset I_\eta
\end{equation}
for all \( n/N \) small enough. Inequalities \eqref{4.2.1} characterize the weighted extremal measure \( \nu_{n,N} \), see \cite[Theorem~I.3.3]{SaffTotik}. Hence, \( \nu_{n,N} \) remains the weighted extremal measure for the weight \( w_N \) even when the minimization problem is posed only on \( I_\eta \). As \( - \log w_N \) is convex on \( I_\eta \), \( \supp(\omega_{n,N}) \) must be an interval, see \cite[Theorem~IV.1.10]{SaffTotik}.

In our case, the support of the weighted extremal measure is an interval while the weight function is absolutely continuous with $L^p$-integrable derivative for some $p>1$. Hence, the extremal distribution is itself absolutely continuous with respect to the Lebesgue measure, see \cite[Theorem~IV.2.4]{SaffTotik}. The proof of the lemma is completed.
\end{proof}

It immediately follows from Lemma~\ref{lem:4.2} that, given \eqref{nN}, we have \( \cap_n \Delta_{n,N}=\{\eta\} \). In fact, more can be said. Recall \eqref{sd1}. Write \( \Delta = [\alpha,\beta] \). To clarify the statement of the following lemma, observe that
\begin{equation}
\label{Ls}
\begin{cases}
\eta_N\in(\alpha,\beta) & \Rightarrow L_N(\eta_N)=0, \; L_N(x)>0, \; x\in[\alpha,\eta_N), \; L_N(x)<0, \; x\in(\eta_N,\beta], \\
\eta_N = \alpha & \Rightarrow L_N(\alpha)\leq 0, \; L_N(x)<0, \; x\in(\alpha,\beta], \\
\eta_N = \beta & \Rightarrow L_N(\beta)\geq 0, \; L_N(x)>0, \; x\in[\alpha,\beta).
\end{cases}
\end{equation}
Moreover, since \( L_N^\prime(x) \to L^\prime(x) \) uniformly on \( \Delta \), \eqref{sd1} gives that $|L_N^\prime(\eta_N)|>\delta$ for some positive $\delta$ and all $N$ large enough.

\begin{lemma} 
\label{lem:4.3}
For $n/N$ sufficiently small, we have that
\begin{equation}
\label{est_Delta_nN}
|\Delta_{n,N}| \left( \frac12|L_N(\eta_N)|+ \frac18 |L_N^\prime(\eta_N)| |\Delta_{n,N}| + \O\big(|\Delta_{n,N}|^2\big) \right) \leq \frac nN\,.
\end{equation}
Write \( \Delta_{n,N}=[\alpha_{n,N},\beta_{n,N}] \). If \( \alpha<\alpha_{n,N} \) and \( \beta_{n,N}<\beta \), then $L_N(\eta_N)=0$ and
\begin{equation}
\label{est_Delta_nN2}
|\Delta_{n,N}| \left(  \frac18 |L_N^\prime(\eta_N)| |\Delta_{n,N}| + \O\big(|\Delta_{n,N}|^2\big) \right) = \frac nN\,.\end{equation}
On the other hand, if \( L_N(\eta_N)\neq 0 \), then   $\eta_N$ is an endpoint of $\Delta$ and
\begin{equation}
\label{est_Delta_nN3}
|\Delta_{n,N}| \left( \frac12|L_N(\eta_N)| + \frac38 |L_N^\prime(\eta_N)| |\Delta_{n,N}| + \O\big(|\Delta_{n,N}|^2\big) \right) = \frac nN\,.
\end{equation}
Furthermore, we have that
\begin{equation}
\label{sd2}
\eta_N 
\begin{Bmatrix}
= \\
\leq \\
\geq 
\end{Bmatrix}
\frac{\alpha_{n,N}+\beta_{n,N}}2 + \mathcal O(|\Delta_{n,N}|^2) ,
\end{equation}
where the first case holds when \( \alpha<\alpha_{n,N} \) and \( \beta_{n,N}<\beta \), the second one when \( \beta_{n,N}<\beta \), and the third one when \( \alpha<\alpha_{n,N} \).
\end{lemma}
\begin{proof}
Recall that \( \Delta_{n,N}\subset I_\eta \) for all \( n/N \) small enough, see \eqref{4.2.4}. Since \( \eta \) is an interior point of \( I_\eta \) when it is an interior point of \( \Delta \), we can take  \( n/N \) sufficiently small so that if \( \alpha_{n,N} \) and/or \( \beta_{n,N} \) is an interior point of \( \Delta \), then it is also an interior point of \( I_\eta \). Thus, it follows from \cite[Theorem~IV.1.11]{SaffTotik} applied on \( I_\eta \) that
\begin{equation}
\label{4.3.1}
\int_{\Delta_{n,N}} L_N(x) \sqrt{\frac{x-\alpha_{n,N}}{\beta_{n,N}-x}} \frac{dx}{\pi} \geq -\frac nN \qandq   \int_{\Delta_{n,N}} L_N(x) \sqrt{\frac{\beta_{n,N}-x}{x-\alpha_{n,N}}} \frac{dx}{\pi} \leq \frac nN,
\end{equation}
where the first inequality is replaced by equality when \( \beta_{n,N}<\beta \) and the second one becomes equality when \( \alpha<\alpha_{n,N}\). Since \( |\Delta_{n,N}| =o(1) \) as \( n/N \to 0 \), at least one of the conditions \( \alpha<\alpha_{n,N}\), or \( \beta_{n,N}<\beta \), or possibly both takes place. According to condition~$(i)$, it holds that
\begin{equation}
\label{4.3.2}
L_N(x) = L_N(\eta_N) + L_N^\prime(\eta_N)(x-\eta_N) + \O\big(|\Delta_{n,N}|^2\big), \quad x\in\Delta_{n,N},
\end{equation}
where the error term is uniform in \( N \) because \( L_N^{\prime\prime}(x) \to L^{\prime\prime}(x) \) uniformly on \( \Delta \) due to condition $(ii)$. Recall \eqref{Ls}. Substituting  \eqref{4.3.2} into the  first inequality in \eqref{4.3.1} and changing signs give
\begin{equation}
\label{4.3.3}
|\Delta_{n,N}| \left( -\frac12L_N(\eta_N) + \frac{3-4\xi_{n,N}}8 |L_N^\prime(\eta_N)| |\Delta_{n,N}| + \mathcal O\big(|\Delta_{n,N}|^2\big) \right) \leq \frac nN,
\end{equation}
where \( \eta_N =: |\Delta_{n,N}|\xi_{n,N}+\alpha_{n,N}\), \( \xi_{n,N}\in[0,1] \), and the inequality is equality when \( \beta_{n,N}<\beta \). In particular, this proves \eqref{est_Delta_nN3} when \( L_N(\eta_N)<0 \), because in this case \( \eta_N=\alpha \) and \( \beta_{n,N}<\beta \), so that we have equality in \eqref{4.3.3} with \( \xi_{n,N} = 0 \).  On the other hand, the second inequality in \eqref{4.3.1} and \eqref{4.3.2} give
\begin{equation}
\label{4.3.4}
|\Delta_{n,N}| \left( \frac12L_N(\eta_N) + \frac{4\xi_{n,N}-1}8 |L_N^\prime(\eta_N)| |\Delta_{n,N}| + \mathcal O\big(|\Delta_{n,N}|^2\big) \right) \leq \frac nN,
\end{equation}
where the inequality is actually an equality when \( \alpha<\alpha_{n,N} \). This, respectively, proves  \eqref{est_Delta_nN3} when \( L_N(\eta_N)>0 \), because in this case \( \eta_N=\beta \) and \( \alpha<\alpha_{n,N} \), so that we have equality in \eqref{4.3.4} with \( \xi_{n,N} = 1 \). Estimate \eqref{est_Delta_nN}  follows by averaging  \eqref{4.3.3} and \eqref{4.3.4}. It becomes \eqref{est_Delta_nN2} when \( \alpha<\alpha_{n,N} \) and \( \beta_{n,N}<\beta \) because \eqref{4.3.3} and \eqref{4.3.4} are both equalities in this case while \( L_N(\eta_N) =0 \). Finally, by taking the difference of  \eqref{4.3.3} and \eqref{4.3.4}, one of which is equality and the other one is a bound, we get that \( \xi_{n,N} \leq 1/2 + \mathcal O(|\Delta_{n,N}|)\) when \( \beta_{n,N}<\beta \) and that \( \xi_{n,N} \geq 1/2 + \mathcal O(|\Delta_{n,N}|)\) when \( \alpha<\alpha_{n,N} \). Clearly, this is equivalent to the inequalities in \eqref{sd2}. The equality is an immediate consequence of both inequalities.
\end{proof}

For small enough $n/N$, this lemma gives
\begin{equation}\label{sd9}
\frac nN \lesssim |\Delta_{n,N}| \lesssim \sqrt{\frac nN} \qandq |\Delta_{n,N}| = \frac nN \frac{2+o(1)}{|L(\eta)|} \quad \text{if} \quad |L(\eta)|>0.
\end{equation}

In the next result, we study the densities of the weighted equilibrium measures \( \omega_{n,N} \). It is more convenient to do this on a fixed interval. To this end, we denote by \( \hat\omega_{n,N} \) the pull back of \( \omega_{n,N} \) from \( \Delta_{n,N} \) to \( [0,1] \). By  Lemma~\ref{lem:4.2}, \( \hat\omega_{n,N} \) is absolutely continuous with respect to the Lebesgue measure and
\[
\hat\omega_{n,N}^\prime(s) = |\Delta_{n,N}| \omega_{n,N}^\prime(s|\Delta_{n,N}|+\alpha_{n,N}), \quad s\in[0,1].
\]

\begin{lemma} 
\label{lem:4.4}
The family of functions \( \{ \hat\omega_{n,N}^\prime \} \) is uniformly equicontinuous on closed subintervals of \( (0,1) \). Moreover,
\begin{equation}
\label{est_omega_nN}
\frac1{3\pi}\sqrt{s(1-s)} < \hat\omega_{n,N}^\prime(s) <\frac9\pi s^{\gamma_0} (1-s)^{\gamma_1}, \quad s\in [0,1],
\end{equation}
where \( \gamma_0 = -\frac12 \) if \( \alpha=\alpha_{n,N} \) and \( \gamma_0 = \frac12 \) otherwise. Similarly, \( \gamma_1 = -\frac12 \) if \( \beta_{n,N}=\beta \) and \( \gamma_1 = \frac12 \) otherwise.
\end{lemma}
\begin{proof}
As we have mentioned at the end of the proof of Lemma~\ref{lem:4.2}, \cite[Theorem~IV.2.4]{SaffTotik} implies that \( \omega_{n,N} \) is absolutely continuous with respect to the Lebesgue measure. In fact, this theorem also provides an explicit integral formula for its density. We use the symmetry and a trigonometric change of variables to write that formula as
\begin{equation}
\label{4.4.1}
\hat\omega_{n,N}^\prime(s) =\frac1{\pi\sqrt{s(1-s)}}\left( 1 + |\Delta_{n,N}|\frac Nn\fint_0^1 \hat L_N(x) \frac{\sqrt{x(1-x)}}{s-x} \frac{dx}\pi\right),
\end{equation}
where \( \hat L_N(x) := L_N(x|\Delta_{n,N}|+\alpha_{n,N}) \) and \( \fint \) is the integral understood in the sense of the principal value.  Set
\[
\delta_{n,N} := 1+ |\Delta_{n,N}|\frac Nn \int_0^1 \hat L_N(x) \sqrt{\frac x{1-x}}\frac{dx}{\pi}.
\]
By \eqref{4.3.1}, this is a nonnegative quantity and it is equal to \( 0 \) when \( \beta_{n,N}<\beta \). Similarly to \eqref{4.3.3}, plugging \eqref{4.3.2} into the definition of \( \delta_{n,N} \) gives
\begin{equation}
\label{4.4.2}
|\Delta_{n,N}| \left( -\frac12L_N(\eta_N) + \frac{3-4\xi_{n,N}}8 |L_N^\prime(\eta_N)| |\Delta_{n,N}| + \mathcal O\big(|\Delta_{n,N}|^2\big) \right) = \frac nN (1-\delta_{n,N}).
\end{equation}
Hence, by using \eqref{4.3.4} in this identity, we can conclude that
\begin{equation}
\label{4.4.3}
2-\delta_{n,N} \geq |\Delta_{n,N}|^2\frac Nn \frac{|L_N^\prime(\eta_N)| + \mathcal O(|\Delta_{n,N}|)}{4} \quad \Rightarrow \quad \delta_{n,N} \in [0,2).
\end{equation}
Now, we deduce from \eqref{4.4.1} and the definition of \( \delta_{n,N} \) that
\begin{equation}
\label{4.4.4}
\hat\omega_{n,N}^\prime(s) = \frac{\delta_{n,N}}{\pi\sqrt{s(1-s)}} + \frac1\pi\sqrt{\frac{1-s}s}|\Delta_{n,N}|\frac Nn\fint_0^1 \frac{\hat L_N(x)}{s-x} \sqrt{\frac x{1-x}} \frac{dx}\pi.
\end{equation}
If \( \hat L_N(x) \) is replaced by \( 1 \) in the last formula, the singular integral above is equal to \( -1 \), which follows from applying the Plemelj-Sokhotski formula. Hence, we can write
\begin{equation}
\label{4.4.5}
\fint_0^1 \frac{\hat L_N(x)}{s-x} \sqrt{\frac x{1-x}} \frac{dx}\pi  = - \hat L_N(s) - \int_0^1 \frac{\hat L_N(s)-\hat L_N(x)}{s-x} \sqrt{\frac x{1-x}} \frac{dx}\pi.
\end{equation}
The mean-value theorem gives
\[
\hat L_N(s)-\hat L_N(x) = |\Delta_{n,N}|L_N^\prime(\xi_{s,x})(s-x) = |\Delta_{n,N}|(L_N^\prime(\eta_N)+\mathcal O(|\Delta_{n,N}|) )(s-x),
\]
where \( \xi_{s,x}\in\Delta_{n,N} \). Using the above estimate for \( \hat L_N(\xi_{n,N})-\hat L_N(s) \) as well, we get that
\[
\fint_0^1 \frac{\hat L_N(x)}{s-x} \sqrt{\frac x{1-x}} \frac{dx}\pi  = -L_N(\eta_N) - |\Delta_{n,N}| L_N^\prime(\eta_N)\left(s-\xi_{n,N}+\frac12\right) + \mathcal O\big(|\Delta_{n,N}|^2\big).
\]
Furthermore, by solving \eqref{4.4.2} for \( -L_N(\eta_N) - |\Delta_{n,N}| L_N^\prime(\eta_N)\xi_{n,N} \) and recalling that \( L_N^\prime(\eta_N)=-|L_N^\prime(\eta_N)| \), we further get 
\[
|\Delta_{n,N}|\frac Nn \fint_0^1 \frac{\hat L_N(x)}{s-x} \sqrt{\frac x{1-x}} \frac{dx}\pi  = 2(1-\delta_{n,N}) + l_{n,N}\left(s-\frac14\right) + \mathcal O(|\Delta_{n,N}|),
\]
where we use \(l_{n,N} := |\Delta_{n,N}|^2(N/n)|L_N^\prime(\eta_N)| \) for shorthand. Plugging the above expression into \eqref{4.4.4} now gives\footnote{The appearance of \( s-\frac14 \) is not entirely surprising since \( \int_0^1(s-\frac 14)\sqrt{\frac {1-s}s}ds = 0 \).}
\begin{equation}
\label{4.4.6}
\hat\omega_{n,N}^\prime(s) =\frac{\delta_{n,N}}{\pi\sqrt{s(1-s)}} + \frac1\pi \left[2(1-\delta_{n,N}) + \mathcal O(|\Delta_{n,N}|) + l_{n,N} \left(s-\frac14\right) \right] \sqrt{\frac{1-s}s}.
\end{equation}

Since \( l_{n,N}\geq0 \) and $\delta_{n,N}\in [0,2)$, see \eqref{4.4.3}, we get for \( s\in[1/4,1] \) that
\[
\hat\omega_{n,N}^\prime(s) \geq \frac1\pi\left(2-\frac23\delta_{n,N}+\mathcal O(|\Delta_{n,N}|)\right)\sqrt{s(1-s)} > \frac1{3\pi}\sqrt{s(1-s)}\,.
\]
In the other direction,  we can rewrite \eqref{4.4.6} as
\[
\hat\omega_{n,N}^\prime(s) = \frac1\pi\frac{\delta_{n,N}s+(2-\delta_{n,N})(1-s)}{\sqrt{s(1-s)}} + \frac1\pi \left[ \mathcal O(|\Delta_{n,N}|) +   l_{n,N} \left(s-\frac14\right) \right] \sqrt{\frac{1-s}s}.
\]
By \eqref{est_Delta_nN}, we have \( l_{n,N}\leq 8+ \mathcal O(|\Delta_{n,N}| )\) and therefore
\[
\hat\omega_{n,N}^\prime(s) \leq \frac2\pi\frac1{\sqrt{s(1-s)}} + \frac{6+\mathcal O(|\Delta_{n,N}|)}\pi \sqrt{\frac{1-s}s} < \frac9\pi\frac1{\sqrt{s(1-s)}}
\]
for all \( s\in[0,1] \). Similarly, when \( \beta_{n,N}<\beta \), \( \delta_{n,N}=0 \) and  \eqref{4.4.6} gives 
\[
\hat\omega_{n,N}^\prime(s)  \leq \frac{8s+\mathcal O(|\Delta_{n,N}|)}\pi\sqrt{\frac{(1-s)}s} < \frac9\pi \sqrt{s(1-s)}
\]
for \( s\in[1/4,1] \). The last three estimates yield \eqref{est_omega_nN} on \( [1/4,1]\). The estimates for  \(s\in  [0,3/4] \) can be obtained analogously.

It only remains to show that the functions \( \hat\omega_{n,N}^\prime \) form a uniformly equicontinuous family on each closed subinterval of \( (0,1) \). To this end, it is enough to show that the functions \( \hat\omega_{n,N}^{\prime\prime} \) are uniformly bounded on each such subinterval. Using \eqref{4.4.5}, we get that
\[
\frac d{ds}\left( \fint_0^1 \frac{\hat L_N(x)}{s-x} \sqrt{\frac x{1-x}} \frac{dx}\pi \right) = - \hat L_N^\prime(s) + \frac12 \int_0^1 \hat L_N^{\prime\prime}(\tilde \xi_{s,x})\sqrt{\frac x{1-x}} \frac{dx}\pi
\]
for some \( \tilde \xi_{s,x}\in \Delta_{n,N} \) by Taylor approximation theorem. Since \( L_N^{(k)} \to L^{(k)} \) uniformly on \( \Delta \) for \( k=1,2 \), \( \hat L_N^{(k)}(s) = |\Delta_{n,N}|^k L_N^{(k)}(s|\Delta_{n,N}|+ \alpha_{n,N}) \), and \( |\Delta_{n,N}|^2(N/n) \lesssim 1 \) while \( \delta_{n,N} \in [0,2) \), the desired claim readily follows from  \eqref{4.4.4}.
\end{proof}

\subsection{Auxiliary Lemmas}

In what follows, we are always assuming that \eqref{nN} takes place. Hence, to simplify the notation, we rename \( \Delta_{n,N}=[\alpha_{n,N},\beta_{n,N}] \) as \( \Delta_n=[\alpha_n,\beta_n] \). Also, all constants \( C_i,C_i^\prime, C_i^{\prime\prime} \), \( i\geq 1\), appearing below as well as implicit constants in inequalities \( \lesssim,\gtrsim\) are understood to be independent of \( n \).

To prove Theorem~\ref{thm:5} it will be convenient to rescale the problem so that the equilibrium measures are supported on the interval \( [0,1] \). To this end, given measures \( \mu_n \) as in Theorem~\ref{thm:5}, we define measures \( \hat\mu_n \) by
\begin{equation}
\label{munhat}
\hat\mu_n(B) = \mu_n^{-1}(|\Delta_n|)\mu_n\big(\big\{x|\Delta_n|+\alpha_n:x\in B\big\} \big)
\end{equation}
for any  set \( B \subseteq \hat \Delta_n =\big [\frac{\alpha-\alpha_n}{|\Delta_n|},\frac{\beta-\alpha_n}{|\Delta_n|} \big ] =: [\hat\alpha_n,\hat\beta_n]\).

\begin{lemma}
\label{lem:4.5}
There exists a constant \( C_1>0 \) such that
\begin{equation}
\label{sada1}
|\hat\mu_n| \lesssim \exp\left(C_1|\Delta_n|^{-1/2}\right).
\end{equation}
Moreover, if we write \( d\hat\mu_{n|[0,1]} = \hat v_{n,[0,1]}d\omega_{[0,1]} + d\hat\mu^s_{n|[0,1]}\), where \( \hat\mu^s_{n|[0,1]} \) is singular with respect to the Lebesgue measure, then
\begin{equation}
\label{sada2}
\int_0^1 |\log\hat v_{n,[0,1]}(x)|d\omega_{[0,1]}(x) \lesssim |\Delta_n|^{-\gamma}.
\end{equation}
\end{lemma}
\begin{proof}
We recall that $\lim_{n\to\infty}\alpha_n=\lim_{n\to\infty}\beta_n=\lim_{n\to\infty}\eta_n=\eta$. Write \( I_\eta=[\alpha_\eta,\beta_\eta] \).  Assume that \( \eta<\beta \). In this case \( \eta<\beta_\eta \). By dropping the singular part and using Jensen's inequality we get that
\begin{align*}
\mu_n(|\Delta_n|) &\geq  \int_{\frac12(\alpha_n+\beta_n)}^{\beta_n}\mu^\prime_n(x) dx \geq \frac{|\Delta_n|}2 \exp\left( \frac2{|\Delta_n|} \int_{\frac12(\alpha_n+\beta_n)}^{\beta_n}\log \mu^\prime_n(x) dx \right) \\
 & \geq \frac{|\Delta_n|}2\exp\left( -\frac2{|\Delta_n|} \int_{\frac12(\alpha_n+\beta_n)}^{\beta_n}|\log^- \mu^\prime_n(x)|dx \right).
\end{align*}
Since it trivially holds that \( \sqrt{(x-\alpha_n)(\beta_\eta-x)} \leq \sqrt{|I_\eta||\Delta_n|} \) for \( x\in \big[\frac12(\alpha_n+\beta_n),\beta_n\big] \), we have 
\begin{align*}
\mu_n(|\Delta_n|) & \geq \frac{|\Delta_n|}2\exp\left( -2\sqrt{\frac{|I_\eta|}{|\Delta_n|}} \int_{\frac12(\alpha_n+\beta_n)}^{\beta_n}|\log^- \mu^\prime_n|d\omega_{[\alpha_n,\beta_\eta]} \right) \\
& \geq \frac{|\Delta_n|}2\exp\left( -2\sqrt{\frac{|I_\eta|}{|\Delta_n|}} \int_{\alpha_n}^{\beta_\eta}|\log^- \mu^\prime_n|d\omega_{[\alpha_n,\beta_\eta]} \right).
\end{align*}
Since \( \alpha_n\to\eta < \beta_\eta \) as \( n\to\infty \), we have that \( \beta_\eta - \alpha_n \gtrsim 1 \). Thus,  \eqref{cond_szego} provides
\[
\int_{\alpha_n}^{\beta_\eta}|\log^- \mu^\prime_n|d\omega_{[\alpha_n,\beta_\eta]} \leq \int_{\alpha_n}^{\beta_\eta}|\log \mu^\prime_n|d\omega_{[\alpha_n,\beta_\eta]} \lesssim 1.
\]
Combining these bounds, we get
\begin{equation}
\label{4.5.1}
\mu_n(|\Delta_n|) \geq \frac{|\Delta_n|}2\exp\left( -C_1^\prime |\Delta_n|^{-1/2} \right) \geq \exp\left( -C_1 |\Delta_n|^{-1/2} \right)
\end{equation}
for some constants \( C_1^\prime,C_1>0 \). Since the measures \( \mu_n \) have uniformly bounded masses, the above estimate and \eqref{munhat} yield \eqref{sada1} in the considered case. If \( \eta=\beta \), the above estimate can be easily modified by considering intervals \( \big[\alpha_n,\frac12(\alpha_n+\beta_n)\big] \) and \( [\alpha_\eta,\beta_n]\) to deduce the same conclusion.

Recall that the arcsine distribution is invariant under the linear transformations. Therefore,  \eqref{munhat} and \eqref{4.5.1} yield 
\[
\int_0^1 |\log\hat v_{n,[0,1]}|d\omega_{[0,1]} \leq \int_{\Delta_n} |\log v_{n,\Delta_n}|d\omega_{\Delta_n} + C_1|\Delta_n|^{-1/2}.
\]
The above estimate and \eqref{cond_szego} readily yield \eqref{sada2} finishing the proof.
\end{proof}

\begin{lemma}
\label{lem:4.6}
Let \( \hat w_n(x) := w_N(x|\Delta_n|+\alpha_n) \),  \( x \in \hat\Delta_n = [\hat\alpha_n,\hat\beta_n] \). Then, there exist a constant \( C_2>0 \) such that
\begin{equation}
\label{sada3}
\hat w_n(x) \begin{cases}
\ge e^{-C_2n/N}, & x\in [0,1], \\
\le  e^{-C_2|x-\frac 12|n/N}, & x\in\hat\Delta_n\setminus[0,1],
\end{cases}
\end{equation}
for all sufficiently large \( n \).

\end{lemma}
\begin{proof}
Using Taylor's theorem at the point $\eta_N$ and  our assumption that $w_N(\eta_N)=1$, we get that
\begin{align*}
\log w_N(x) &= L_N(\eta_N)(x-\eta_N) + \O\left((x-\eta_N)^2\right) \\
& \geq -|L_N(\eta_N)||\Delta_n| + \O\big(|\Delta_n|^2\big) \stackrel{\eqref{est_Delta_nN}}{\gtrsim} -(n/N+\O(|\Delta_n|^2)) \stackrel{\eqref{sd9}}{\gtrsim} -n/N
\end{align*}
for \( x\in \Delta_n \), which is equivalent to the top line in \eqref{sada3}.

To prove the upper bound in \eqref{sada3}, assume first that \( 1<\hat\beta_n \), i.e., \( \beta_n<\beta \). In this case, \( L_N(\eta_N)\leq 0 \). Recall \eqref{sd1}. Let \( c>0 \) be such that \( L_N^\prime(x)\leq-2c \) for \( x\in I_\eta \). Recall also that \( w_N(\eta_N) = 1 \). Therefore, Taylor's theorem applied at the point $\eta_N$ gives
\begin{align}
\log w_N(x) & \leq L_N(\eta_N)(x-\eta_N) - c(x-\eta_N)^2, & \eta_N\leq x\in I_\eta, \nonumber \\
& \lesssim L_N(\eta_N)(x-\eta_N) - c(x-\eta_N)^2, & \eta_N\leq x\in \Delta,
\label{4.6.1}
\end{align}
where the second inequality takes place because \( w_N \) decreases on \( (\eta_N,\beta] \) and \( I_\eta,\Delta \) are fixed intervals independent of $n$.  Consider two cases. 

{\bf 1.} First, assume that 
\begin{equation}
\label{sd11}
L_N(\eta_N) \leq - \frac34|L_N^\prime(\eta_N)||\Delta_n|,
\end{equation} 
in which case \( \eta_N=\alpha \) and \eqref{est_Delta_nN3} takes place. Then, 
\begin{multline}
\label{sad10}
\frac34|L_N^\prime(\eta_N)| |\Delta_n|^2 + \mathcal O\big(|\Delta_n|^3\big) \stackrel{\eqref{sd11}}{\leq} \\ 
\frac12|L_N(\eta_N)||\Delta_n| + \frac38|L_N^\prime(\eta_N)| |\Delta_n|^2 + \mathcal O\big(|\Delta_n|^3\big) \stackrel{\eqref{est_Delta_nN3}}{=} \frac nN+\mathcal O\big(|\Delta_n|^3\big).
\end{multline}
Using \eqref{est_Delta_nN3} one more time together with the above estimates gives
\begin{align*}
L_N(\eta_N)|\Delta_n| &= -2\frac nN + \frac34|L_N^\prime(\eta_N)| |\Delta_n|^2 + \mathcal O\big(|\Delta_n|^3\big) \\
& \stackrel{\eqref{sad10}}{\leq} -\frac nN + \mathcal O\big(|\Delta_n|^3\big) \le -\frac n{2N}
\end{align*}
for all large $n$. Thus, \eqref{4.6.1} and the above inequality give us
 \[
w_n(x) \leq \exp\big(2C_2^\prime L_N(\eta_N)(x-\eta_N)\big) \leq \exp\left(-C_2^\prime \frac nN\frac{x-\eta_N}{|\Delta_n|} \right), \quad x\in[\beta_n,\beta],
 \]
 for some constant \( C_2^\prime>0 \). Now, notice that our assumption $\beta_n<\beta$ implies  $\eta_N\le (\alpha_n+\beta_n)/2+\mathcal O(|\Delta_n|^2)$ due to \eqref{sd2}.  Hence, after rescaling, we obtain that
 \begin{equation}
\label{4.6.2}
\hat w_n(x) \le \exp\big(-C_2(n/N)|x-\tfrac 12| \big), \quad x\in[1,\hat\beta_n],
\end{equation}
for some constant \( C_2>0 \).

{\bf 2.}  Suppose now that 
\begin{equation}
\label{sd12}
 -\frac34 |L_N^\prime(\eta_N)||\Delta_n|<L_N(\eta_N).
\end{equation} 
Then, since in that case \( \beta_n<\beta \), we have equality in \eqref{4.3.3}, so
\[
\frac32|L_N^\prime(\eta_N)| |\Delta_n|^2 \stackrel{\eqref{sd12}}{\geq} |L_N(\eta_N)||\Delta_n| + \frac34|L_N^\prime(\eta_N)| |\Delta_n|^2   \stackrel{\eqref{4.3.3}}{\geq} 2\frac nN+\mathcal O(|\Delta_n|^3)
\]
for all sufficiently large \(n\). Because \( |L_N^\prime(\eta_N)|\to |L^\prime(\eta)|>0\), the previous bound gives $|\Delta_n|^2\gtrsim n/N$. Hence, we keep the second term in \eqref{4.6.1} to get 
\[
w_n(x) \leq \exp\left( -cC_2^\prime(x-\eta_N)^2 \right) \leq \exp\left(-C_2^{\prime\prime} \frac nN\left(\frac{x-\eta_N}{|\Delta_n|}\right)^2\right), \quad x\in[\beta_n,\beta],
\]
 for some constant \( C_2^{\prime\prime}>0 \). Again, since $\eta_N\le (\alpha_n+\beta_n)/2+\mathcal O(|\Delta_n|^2)$, we  rescale and possibly adjust constant \( C_2 \)  to deduce
\begin{equation}
\label{4.6.3}
\hat w_n(x) \leq \exp\left( -2C_2(n/N)(x-\tfrac 12)^2 \right), \quad x\in[1,\hat\beta_n].
\end{equation}
Clearly, the bottom line of \eqref{sada3} on \( [1,\hat\beta_n] \) follows from \eqref{4.6.2} and \eqref{4.6.3}. The desired estimate for \( x<0 \) when \( \alpha< \alpha_n \) can be proven analogously.
\end{proof}

\begin{lemma}
\label{lem:4.7}
There exists a constant \( C_3>0 \) such that
\begin{equation}
\label{sada4}
\int_0^1 |P(x)| \hat w_n^{2N}(x)d\hat\mu_n(x) \geq \exp\left(-C_3(n+|\Delta_n|^{-\gamma})\right)
\end{equation}
for any monic polynomial \( P \) of degree at most \( 2n \).
\end{lemma}
\begin{proof}
Apply the top line in \eqref{sada3}. We drop the singular part of \( \hat\mu_n \), use Jensen's inequality and then \eqref{sada2} to write
\begin{align*}
\int_0^1 |P| \hat w_n^{2N}d\hat\mu_n  & \stackrel{\eqref{sada3}}{\geq} e^{-2C_2n}\int_0^1 |P|\hat v_{n,[0,1]} d\omega_{[0,1]} \\
& \stackrel{\text{\rm (Jensen)}}{\geq} \exp\left(-2C_2n + \int_0^1 \log\big(|P \hat v_{n,[0,1]}\big) d\omega_{[0,1]}\right) \\
 & \stackrel{\eqref{sada2}}{\geq} \exp\left( -C_3^\prime \big(n + |\Delta_n|^{-\gamma} \big) + \int_0^1 \log|P| d\omega_{[0,1]}\right)
\end{align*}
for some constant \( C_3^\prime>0 \). Recall \eqref{phi-w}. Observe that \( z\phi_{[0,1]}(z) \to 1/4\)  as \( z\to\infty \). Then, we get from the mean-value inequality for subharmonic functions that
\begin{align}
\int_0^1 \log|P| d\omega_{[0,1]}  = \int_0^1 \log|P\phi_{[0,1]}^d| d\omega_{[0,1]}  \geq \log|P\phi_{[0,1]}^d|(\infty) = -d\log 4,
\end{align}
where \( d=\deg P \). As \( d\leq 2n \), this proves \eqref{sada4} with \( C_3 = C_3^\prime + 4\log2\).
\end{proof}

\begin{lemma}
\label{lem:4.8}
Let \( Z(T_n) \) be the set of zeros of the polynomial \(  T_n := T_n(\hat w_n^{2N}\hat\mu_n) \). Then, there exists a constant \( C_4>0 \) such that 
\begin{equation}
\label{sada5}
Z(T_n) \subset \left\{x:|x-\tfrac 12| \leq C_4 \max\left\{1,n^{-1}|\Delta_n|^{-\gamma}\right\} \right\}.
\end{equation}
for all large enough  \( n \).
\end{lemma}
\begin{proof}
Let \( \{|x-\tfrac 12|\leq x_n \} \) be the smallest interval symmetric around the point \( \tfrac 12 \) that contains \( Z(T_n) \).  Denote by \( Q \) the monic polynomial whose zeros are precisely those zeros of \( T_n \) that belong to \( \{|x-\tfrac 12|\leq \tfrac12x_n \} \), counting multiplicities. If there are no such zeros, we let $Q=1$. Necessarily, \( \deg(Q)\leq n-1 \). Then, the orthogonality condition gives us
\begin{equation}
\label{4.8.1}
\int_{\hat\Delta_n} T_nQ \hat w_n^{2N}d\hat\mu_n =0. 
\end{equation}
We can assume without loss of generality that \( x_n\geq 1 \) for infinitely many indices \( n \) as otherwise we have nothing to prove. In what follows, we are only interested in such indices. Notice that \( (T_nQ)(x)  \) has a constant sign on the interval \( \{|x-\tfrac 12|\leq \tfrac12x_n \}\), which contains \( [0,1] \). Therefore,  \eqref{sada4} yields
\begin{equation}
\label{4.8.2}
\left|\int_{\big\{|x-\frac 12|\leq \frac12x_n \big\}} T_nQ \hat w_n^{2N}d\hat\mu_n \right| \geq \exp\left(-C_3(n+|\Delta_n|^{-\gamma})\right).
\end{equation}
On the other hand, in view of \eqref{sada3}, there exists \( x_*>1 \) and \( C_4^\prime>0 \) such that 
\begin{equation}\label{sd14}
2n\log (4x) + 2N \log \hat w_n(x) \leq -C_4^\prime n (x-\tfrac 12), \quad x\geq x_*.
\end{equation}
Again, we only need to consider those indices \( n \) for which \( x_n+1\geq 2x_* \). Since \( T_n \) and $Q$ are monic and have all  zeros on \( [-x_n+\frac12,x_n+\frac12] \), we can write a bound
\[
|T_n(x)Q(x)|\le |x+x_n-\tfrac 12|^{2n}, \, x>\tfrac 12.
\]
Each function \( \hat w_n(x) \) is decreasing in \( x\) on \( (1,\infty) \), so  \eqref{sada1} and \eqref{sd14} give us
\[
\int_{\frac12(x_n+1)}^{x_n+\frac12} |(T_nQ)| \hat w_n^{2N}d\hat\mu_n \lesssim e^{-C_4^\prime nx_n+C_1|\Delta_n|^{-1/2}}.
\]
Similarly, for every \( j\in\{0,1,2,\ldots\} \) satisfying \( x_n+j+\frac12<\hat\beta_n \), we get 
\[
\int_{x_n+j+\frac12}^{x_n+j+\frac32} |T_nQ| \hat w_n^{2N}d\hat\mu_n \lesssim e^{-C_4^\prime n(x_n+j)+C_1|\Delta_n|^{-1/2}}.
\]
Analogous estimates can also be derived on the intervals \( [\frac12-x_n,\frac{1-x_n}2] \) and \( [-\frac12-j-x_n,\frac12-j-x_n] \). Combining these bounds, we get
\begin{equation}
\label{4.8.3}
\int_{\hat\Delta_n\setminus \big\{|x-\frac 12|\leq \frac12x_n \big\}} |(T_nQ)| \hat w_n^{2N}d\hat\mu_n \leq e^{-C_4^{\prime\prime} n x_n+C_1|\Delta_n|^{-1/2}}
\end{equation}
for some constant \( C_4^{\prime\prime}>0  \). However, if
\[
x_n > \frac{C_1+C_3}{C_4^{\prime\prime}} \frac{1}{n|\Delta_n|^\gamma},
\]
estimates \eqref{4.8.2} and \eqref{4.8.3}  contradict \eqref{4.8.1}. Thus, \( x_n \) is bounded above by the maximum between \( 2x_*-1 \) and the right-hand side of the above inequality, which is exactly the statement of the lemma.
\end{proof}

Recall \eqref{4.2.1} and that \( \supp(\omega_n) \) is an interval \( \Delta_n \) for all \( n \) large enough. Therefore,  \cite[Theorem~I.4.8]{SaffTotik} implies that \( V^{\omega_n} \) is continuous everywhere in \( \C \). In particular, the first inequality in \eqref{4.2.1} holds everywhere on \( \Delta \) and not just quasi-everywhere. Thus, in the rescaled variable, we have that
\begin{equation}
\label{Robin_nN}
N\log \hat w_n(x) \leq nV^{\hat\omega_n(x)}  -\hat \ell_n, \quad  x\in\hat\Delta_n,
\end{equation}
where \( \hat \ell_n := n\log|\Delta_n| + N\ell_{n,N} \) and the that inequality is actually equality on \( [0,1] \).

\begin{lemma}
\label{lem:4.9}
For all large enough \( n \), the bound
\begin{equation}
\label{sada6}
\int_{\hat\Delta_n} |T_n(x)|^2 \hat w_n^{2N}(x)d\hat\mu_n(x) \lesssim \exp\left(-2\hat\ell_n+C_1|\Delta_n|^{-1/2}\right)
\end{equation}
holds.
\end{lemma}
\begin{proof}
Suppose we are given a sequence of probability measures  \( d\sigma_n(x) = u_n(x)dx \) on \( [0,1] \) such that the set \( \{ u_n \} \) is uniformly equicontinuous on closed subintervals of \( (0,1)\) and 
\[
\frac 1A \big(x(1-x)\big)^{b^\prime} \leq u_n(x) \leq A\big(x(1-x)\big)^{b^{\prime\prime}}, \quad x\in(0,1),
\]
for every \( n \) and some constants \( A>1\) and \(b^\prime,b^{\prime\prime}>-1 \). In \cite[Lemma~9.1]{Totik}, it was shown that for sufficiently large \( n \) there exist monic polynomials \( Q(\sigma_n;z) \) of degree \( n \) (see the display right after  \cite[Equation (9.7)]{Totik}), such that
\begin{equation}
\label{4.9.1}
1 \leq |Q(\sigma_n;x)|e^{nV^{\sigma_n}(x)} \leq \min\left\{ n^B, \big(x(1-x)\big)^{-B} \right\}, \quad x\in(0,1),
\end{equation}
for some constant \( B \) depending only on \( A,b^\prime,b^{\prime\prime} \). We can always take $B$ to be an integer.

According to Lemma~\ref{lem:4.4}, the measures \( \sigma_n = \hat\omega_{n+2B} \) satisfy all the above requirements. Set
\[
\Psi_n(z) := Q(\sigma_{n-2B};z)\big(z(z-1)\big)^{B}.
\]
One can readily see from the upper bound in \eqref{est_omega_nN} that
\begin{equation}
\label{4.9.2}
V^{\hat\omega_n}(x) < 9 V^{\omega_{[0,1]}}(x) = 9\log 4, \quad x\in [0,1].
\end{equation}
Thus, it follows from \eqref{4.9.1} and \eqref{4.9.2} that
\[
|\Psi_n(x)|e^{nV^{\hat\omega_n}(x)} \leq \frac14e^{2B V^{\hat\omega_n}(x)} < \frac{(9\log 4)^{2B}}4, \quad x\in [0,1].
\]
The maximum principle for subharmonic functions now yields that
\begin{equation}
\label{4.9.3}
|\Psi_n(z)|e^{nV^{\hat\omega_n}(z)} \lesssim 1, \quad z\in\C.
\end{equation}

Because the \( n \)-th monic  orthogonal polynomial minimizes the \( L^2 \)-norm among all monic polynomials of degree \( n \) with respect to the measure of orthogonality, we now get from \eqref{Robin_nN}, \eqref{4.9.3}, and \eqref{sada1} that
\begin{align*}
\int_{\hat\Delta_n} |T_n|^2 \hat w_n^{2N}d\hat\mu_n & \stackrel{\eqref{Robin_nN}}{\leq} e^{-2\hat\ell_n} \int_{\hat\Delta_n} |\Psi_n|^2 e^{2nV^{\hat\omega_n}}d\hat\mu_n \\
& \stackrel{\eqref{4.9.3}}{\lesssim} e^{-2\hat\ell_n } \hat\mu_n(\hat\Delta_n) \stackrel{\eqref{sada1}}{\lesssim} \exp\left( -2\hat\ell_n + C_1|\Delta_n|^{-1/2} \right). \qedhere
\end{align*}
\end{proof}

\subsection{Proof of Theorem~\ref{thm:5}}

Clearly, to prove the theorem, it is sufficient to show that \eqref{totik} holds uniformly on closed sets of the form \( K_{n,\rho} := \{z:\min_{x\in\Delta_n}|z-x|\geq \rho \} \) for each fixed \( \rho>0 \). Let \( T_n \) be as in Lemma~\ref{lem:4.8}. Since the monic orthogonal polynomial does not depend on the normalization of the measure of orthogonality, it holds that
\[
T_n\big(w_N^{2N}\mu_n\big)(z) = |\Delta_n|^n T_n\left(\frac{z-\alpha_n}{|\Delta_n|}\right).
\]
Hence,  \eqref{totik} is equivalent to saying that the asymptotics
\begin{equation}
\label{4.1.1}
T_n(z) = (1+o(1))\exp\left( n\int\log(z-s)d\hat\omega_n(s)\right)
\end{equation}
holds uniformly on closed set \(  \hat K_{n,\rho} := \{z:\min_{x\in[0,1]}|z-x|\geq \rho/|\Delta_n| \} \) for every $\rho>0$.

Let \( G_n(z) := G(\hat\mu_{n|[0,1]};z) \) be the Szeg\H{o} function of \( \hat\mu_{n|[0,1]} \), see \eqref{SzegoFun}. That is, \( G_n \) is an outer function in \( H^2(D_{[0,1]}) \) such that
\[
|G_{n\pm}(x)|^2= \hat v_{n,[0,1]}(x) \quad \text{for almost every} \quad x\in(0,1),
\]
see Lemma~\ref{lem:4.5}. One can readily verify that
\[
\left|\frac{\sqrt{z(z-1)}}{z-x}-1\right| \lesssim \frac{|\Delta_n|}\rho, \quad z\in \hat K_{n,\rho},
 \]
for all sufficiently large \(n\). Hence, it follows from \eqref{outer}, \eqref{geom-mean}, \eqref{SzegoFun}, and \eqref{sada2} that
\begin{equation}
\label{4.1.2}
\frac{G_n(z)}{G_n(\infty)} = \exp\left( \frac12 \int_0^1 \left(\frac{\sqrt{z(z-1)}}{z-x}-1\right)\log\hat v_{n,[0,1]}(x)d\omega_{[0,1]}(x)   \right)  = 1 + o(1)
\end{equation}
for \( z\in \hat K_{n,\rho} \) as \( n\to\infty \). To prove \eqref{4.1.1}, define
\[
g_n(z) := \frac{T_n(z)}{\|T_n\|_{L^2(\hat w_n^{2N}\hat\mu_n)}}\exp\left( -\hat\ell_n-n\int\log(z-s)d\hat\omega_n(s) \right) G_n(z).
\]
In view of \eqref{4.1.2}, \eqref{4.1.1} will follow if we show that
\begin{equation}
\label{4.1.3}
g_n(z)/g_n(\infty) = 1 + o(1), \quad z\in \hat K_{n,\rho},
\end{equation}
as \( n\to\infty \). Clearly, each \( g_n \) is a function in \( H^2(D_{[0,1]})\). In \eqref{Robin_nN}, we have  equality on \( [0,1] \),  so
\[
\int|g_{n\pm}|^2d\omega_{[0,1]} = \|T_n\|_{L^2(\hat w_n^{2N}\hat\mu_n)}^{-2} \int |T_n|^2\hat w_n^{2N} \hat v_{n,[0,1]}d\omega_{[0,1]} \leq 1.
\]
Moreover, as we already observed above, \eqref{geom-mean} and \eqref{sada2} imply that
\[
G_n(\infty) = \exp\left(\frac12\int\log \hat v_{n,[0,1]}d\omega_{[0,1]}\right) \geq \exp\left( -C_5|\Delta_n|^{-\gamma} \right)
\]
for some constant \( C_5>0 \). Since \( \gamma>1/2 \),  the bound \eqref{sada6} gives
\[
g_n(\infty) = \frac{e^{-\hat\ell_n}G_n(\infty)}{\|T_n\|_{L^2(\hat w_n^{2N}\hat\mu_n)}} \geq \exp\left( -C_6|\Delta_n|^{-\gamma} \right)
\]
for some constant \( C_6>0 \).

Next, we will use a conformal map to the unit disk $\mathbb{D}$ to apply the standard results of the Hardy spaces theory on $\mathbb{D}$. Define
\[
\boldsymbol g_n(z) := g_n\big(\phi_{[0,1]}^{(-1)}(z)\big), \quad z\in\D,
\]
where \( \phi_{[0,1]}^{(-1)}(z) \) denotes the inverse of \( \phi_{[0,1]}(z) \), the conformal map introduced in \eqref{phi-w}.  To show \eqref{4.1.3}, it is enough to check that
\begin{equation}
\label{4.1.4}
\boldsymbol g_n(z)/\boldsymbol g_n(0) = 1 + o(1), \quad |z| \leq |\Delta_n|/(2\rho),
\end{equation}
as \( n\to\infty \). Naturally, \( \boldsymbol g_n(z) \) is a function in \( H^2(\D) \) such that
\begin{equation}
\label{4.1.5}
\frac1{2\pi}\int_{\T}|\boldsymbol g_n(\xi)|^2|d\xi| \leq 1 \qandq \boldsymbol g_n(0) \geq \exp\left( -C_6|\Delta_n|^{-\gamma} \right).
\end{equation}
We can factor $\boldsymbol g_n$ as
\begin{equation}
\label{4.1.6}
\boldsymbol g_n(z) = \boldsymbol w_n(z) \boldsymbol b_n(z), \quad \boldsymbol b_n(z) = \iota_n\prod_i\frac{z-a_{n,i}}{1-a_{n,i}z},
\end{equation}
where \(\boldsymbol w_n \) is outer and \( \iota_n\in\{\pm1\} \) is chosen so that \( \boldsymbol b_n(0)>0 \) (because polynomials \( T_n \) are obviously continuous in \( \C \) and \( G_n \) is outer in \( H^2(D_{[0,1]}) \), there is no singular inner function in the above decomposition). Observe that the zeros \( a_{n,i} \) are all real and are precisely the images of the zeros \( T_n \) that lie outside of \( [0,1] \) under the map \( \phi_{[0,1]} \).

Since Blaschke products are unimodular on the unit circle, it follows from the first estimate in \eqref{4.1.5} that
\begin{equation}
\label{4.1.7}
0 \leq \frac1{2\pi}\int_\T \log^+|\boldsymbol w_n(\xi)||d\xi| \leq \frac{1}{4\pi}\int_\T \boldsymbol w_n^2(\xi)d|\xi| \le \frac 12.
\end{equation}
On the other hand, since \( |\boldsymbol b_n(z)| \leq 1 \) in \( \D \), it holds that \( \boldsymbol w_n(0) \geq e^{-C_6|\Delta_n|^{-\gamma}} \) by the second estimate in \eqref{4.1.5}. In particular, we have that
\[
0 \geq \frac1{2\pi}\int_\T \log^-|\boldsymbol w_n(\xi)||d\xi| = \log \boldsymbol w_n(0) - \frac1{2\pi}\int_\T \log^+|\boldsymbol w_n(\xi)||d\xi| \gtrsim -|\Delta_n|^{-\gamma}.
\]
Altogether, we see that
\[
0 \leq \frac1{2\pi}\int_\T \big|\log|\boldsymbol w_n(\xi)|\big| |d\xi| \lesssim |\Delta_n|^{-\gamma}.
\]
Therefore, similarly to \eqref{4.1.2}, it holds that
\[
\frac{\boldsymbol w_n(z)}{\boldsymbol w_n(0)} = \exp\left( \frac1{2\pi}\int_\T \left( \frac{\xi+z}{\xi-z} -1\right) \log|\boldsymbol w_n(\xi)||d\xi| \right) = 1+ \mathcal O\big(|\Delta_n|^{1-\gamma} \big)
\]
for \( |z|< |\Delta_n|/(2\rho) \) and all  large enough \(n\). Thus, in view of \eqref{4.1.6} and the above asymptotic formula, to prove \eqref{4.1.4}, it is sufficient to show that
\begin{equation}
\label{4.1.8}
\boldsymbol b_n(z)/\boldsymbol b_n(0) = 1 + o(1), \quad |z| \leq |\Delta_n|/(2\rho), \qasq n\to\infty.
\end{equation}
Let \( x_{n,i} \) be the zero of \( T_n \) corresponding to \( a_{n,i} \). Then, from Lemma~\ref{lem:4.8}, if follows that
\[
\big|a_{n,i}^{-1}-a_{n,i}\big| = 4|x_{n,i}(x_{n,i}-1)|^{1/2} \lesssim \max\left\{1,n^{-1}|\Delta_n|^{-\gamma}\right\}.
\]
We consider two cases. 

{\bf 1.} Assume first that \( n^{-1}|\Delta_n|^{-\gamma}\geq1 \). Then,  for \( |z|\leq |\Delta_n|/(2\rho) \), we get 
\begin{equation}
\label{4.1.9}
|z|\sum_{i} \big|a_{n,i}^{-1}-a_{n,i}\big| \lesssim |\Delta_n|^{1-\gamma}
\end{equation}
because the sum above has at most \( n \) terms. 

{\bf 2.} Now, suppose that \( n^{-1}|\Delta_n|^{-\gamma}<1 \). In particular, this implies that the zeros \( a_{n,i} \) are all separated away from zero by a constant independent of \( n \).  Since $|\boldsymbol w_n(\xi)|=|\boldsymbol g_n(\xi)|$ for $\xi\in \mathbb{T}$, the mean-value inequality and bound \eqref{4.1.5} imply
\[
 |\boldsymbol w_n(0)|^2\leq \frac1{2\pi}\int_{\T}|\boldsymbol w_n(\xi)|^2|d\xi|= \frac1{2\pi}\int_{\T}|\boldsymbol g_n(\xi)|^2|d\xi|\stackrel{\eqref{4.1.5}}{\leq} 1.
\]
Therefore, \( \boldsymbol b_n(0) \geq e^{-C_6|\Delta_n|^{-\gamma}} \) by the second estimate in \eqref{4.1.5}. Respectively, we have for \( |z|\leq |\Delta_n|/(2\rho) \) that
\begin{align}
|z|\sum_{i} \big|a_{n,i}^{-1}-a_{n,i}\big| & \lesssim |z| \sum_{i} (1-|a_{n,i}|) \leq -|z|\sum_{i} \log|a_{n,i}| \nonumber \\
& = -|z| \log\boldsymbol b_n(0) \lesssim |\Delta_n|^{1-\gamma}.
\label{4.1.10}
\end{align}
Estimates \eqref{4.1.9} and \eqref{4.1.10} justify the first inequality below and give
\[
\left| \sum_i\log\left(1-z\frac{a_{n,i}^{-1}-a_{n,i}}{1-a_{n,i}z}\right)\right| \lesssim |z|\sum_i  \big|a_{n,i}^{-1}-a_{n,i}\big| \lesssim |\Delta_n|^{1-\gamma}.  
\]
It is straightforward to see that the above estimates yield \eqref{4.1.8}, which finishes the proof of the theorem.
\small

\bibliographystyle{plain}

\bibliography{vw}

\end{document}